\documentclass[12pt,a4paper]{article}
\usepackage[utf8]{inputenc}
\usepackage[english]{babel}
 \usepackage{mathrsfs}
\usepackage[T1]{fontenc}
\usepackage{amsmath}
\usepackage{thmtools}
\usepackage{thm-restate}
\usepackage{amsfonts}
\usepackage{amssymb}
\usepackage{makeidx}
\usepackage{color}
\usepackage{amsthm}
\usepackage{graphicx}
\usepackage[left=2cm,right=2cm,top=2cm,bottom=2cm]{geometry}
\author{Dang Nguyen-Bac, Favre Charles}
\usepackage[all]{xy}

\newtheorem{thmint}{Theorem}

\newtheorem{thm}{Theorem}[section]
\newtheorem{lem}[thm]{Lemma}
\newtheorem{prop}[thm]{Proposition}
\newtheorem{prop-def}[thm]{Proposition-Definition}
\newtheorem{cor}[thm]{Corollary}
\newtheorem{defi}[thm]{Definition}
\newtheorem{rem}[thm]{Remark}
\newtheorem{ex}[thm]{Example}

\newtheorem{conj}{Conjecture}

\theoremstyle{remark}

  \newcommand{\numbpf}[1]{\num_{\BPF}^{#1}}
  
  \newcommand{\normbpf}[1]{\left\lVert#1\right\rVert_{\BPF}}
  \newcommand{\normsigma}[1]{\left\lVert#1\right\rVert_{\Sigma}}
  \newcommand{\normbpfs}[1]{\left\lVert#1\right\rVert_{\BPF}^\vee}
  \newcommand{\normom}[1]{\left\lVert#1 \right\rVert_{\omega}}

\newcommand{\C}{\mathbb{C}}
\newcommand{\R}{\mathbb{R}}

\newcommand{\Q}{\mathbb{Q}}
\newcommand{\N}{\mathbb{N}}

\newcommand{\Pp}{\mathbb{P}}

\newcommand{\K}{K}

\newcommand{\cM}{ \mathcal{M}}

\newcommand{\X}{ \mathcal{X} }

\newcommand{\cI}{ \mathcal{I} }
\newcommand{\cJ}{ \mathcal{J} }

\newcommand{\cO}{ \mathcal{O} }
\newcommand{\cV}{ \mathcal{V} }

\newcommand{\fA}{\mathfrak{A}}
\newcommand{\fB}{\mathfrak{B}}

\newcommand{\om}{{\omega}}
\newcommand{\fM}{{\mathfrak{M}}}

\DeclareMathOperator{\hvol}{\widehat{\vol}}

\DeclareMathOperator{\ord}{ord}

\DeclareMathOperator{\LX}{LX}

\DeclareMathOperator{\Mv}{Mv}

\DeclareMathOperator{\spec}{spec}

\DeclareMathOperator{\psef}{Psef}

\DeclareMathOperator{\BPF}{BPF}
\DeclareMathOperator{\cBPF}{c-BPF}

\DeclareMathOperator{\Nef}{Nef^1}
\DeclareMathOperator{\cNef}{c-Nef^1}
\DeclareMathOperator{\nef}{Nef}

\DeclareMathOperator{\NS}{NS}

\DeclareMathOperator{\vol}{vol}

\DeclareMathOperator{\nums}{N^1_{\Sigma}}
\newcommand{\numbpfs}[1]{\num^{#1,\vee}_{\BPF}}

\DeclareMathOperator{\Vect}{Vect}

\DeclareMathOperator{\num}{N}

\DeclareMathOperator{\wnum}{w-N}

\DeclareMathOperator{\cnum}{c-N}

\author{Nguyen-Bac Dang, Charles Favre}

\title{Intersection theory of nef $b$-divisor classes}

\begin{document}

\maketitle

\begin{abstract}
We prove that any nef $b$-divisor class on a projective variety defined over an algebraically closed field of characteristic $0$ is a decreasing limit of nef Cartier classes. Building on this
technical result, we construct an intersection theory of nef $b$-divisors, and prove several variants of the Hodge index theorem inspired by the work of Dinh and Sibony.
We show that any big and basepoint free curve class is a power of a nef $b$-divisor, and relate this statement to
Zariski decompositions of curves classes introduced by Lehmann and Xiao.  
Our construction allows us to relate various Banach spaces contained in the space of $b$-divisors which were defined in our previous work. 
\end{abstract}

\setcounter{tocdepth}{2}

\tableofcontents

\medskip

\addcontentsline{toc}{section}{Introduction}

\section*{Introduction}
The notion of $b$-divisor was introduced by Shokurov and have found striking applications in algebraic geometry and singularity theory~\cite{shokurov_3_fold,shokurov_prelimiting,MR2134273,corti_3_fold,
MR2966720,MR2931273,zhang_volume_isolated_singularities,MR3329027,b-flips}, 
 and in algebraic dynamical systems~\cite{boucksom_favre_jonsson_deggrowth,cantat_bir_surfaces,
gignac_ruggiero,xie_periodic,blanc_cantat,spectral}. It provides the right setting to construct Zariski decomposition of divisors~\cite{boucksom_favre_jonsson_volume,MR3010168,MR3368254} and higher codimension cycles \cite{fulger_lehmann_zariski}.
It is also deeply tied with some recent developments in pluripotential analysis on non-Archimedean analytic varieties
(see \cite{MR1311351,MR1629925,MR2244803,MR3419957,trivial-valued,MR4018262} and the references therein).

A $b$-divisor is by definition a collection of Neron-Severi classes in all birational models of a fixed projective variety satisfying natural compatibility conditions under push-forward morphisms. 
It is thus in essence an "infinite" object and defining the intersection product of two such objects always requires a limiting process which may be guaranteed  only under special circumstances. 

Such an intersection theory have been successfully developed for special classes of $b$-divisors:  Cartier $b$-divisors in~\cite{MR3239119}, relatively nef $b$-divisors over 
a closed point in~\cite{MR2426355}, and over the spectrum of a valuation ring~\cite{MR3419957,trivial-valued,MR4018262}. In this paper, we focus on nef  $b$-divisors
over a projective variety, and show that one can develop a natural intersection calculus for them. We proceed by proving
a Hodge index theorem in this context, and solve the operator $\alpha \mapsto \alpha^{d-1}$ ($d$ being the dimension of the ambiant variety) which may be viewed as
an analog of the complex Monge-Amp\`ere operator. 

\begin{center} $\diamond$ \end{center}

We work under the following setup, and borrow the terminology from~\cite{boucksom_favre_jonsson_deggrowth,spectral}.

Let $X$ be any smooth projective variety of dimension $d$ defined over an algebraically closed field $K$ of characteristic $0$.
A Weil $b$-divisor class $\alpha$ is a family of real Neron-Severi classes $\alpha_{X'}\in\NS(X')$ that are compatible under push-forward. 
Here $X'$ runs over all smooth birational models lying over $X$.

The space $\wnum^1(\X)$ of Weil $b$-divisor classes forms an infinite dimensional vector space. It can be identified
with the projective limit of the real Neron-Severi spaces of all models, hence carries a natural locally convex topology
in which every bounded set is relatively compact.

Cartier $b$-divisors are Weil $b$-divisors for which there exists a model 
$X'$ such that $\alpha_{X''} = \pi^* \alpha_{X'}$ for any model such that $\pi\colon X''\to X'$ is regular. In this situation
we say that $\alpha$ is determined by $\alpha_{X'}$ in $X'$. The space $\cnum^1(\X)$ of Cartier $b$-divisors is dense in $\wnum^1(\X)$.

Since the nef property is stable by pull-back one can define nef Cartier classes as those determined by nef classes in some model.
We say that a class $\alpha \in\wnum^1(\X)$ is nef if there exists a sequence of nef Cartier $b$-classes such that $\alpha_n \to \alpha$.
Our main aim is to explore in depth the structure of the convex cone $\Nef(\X)$ of nef $b$-divisors. 
To that end, we first prove a fundamental approximation result for arbitrary nef classes.

\begin{restatable}{thmint}{approx}\label{thm:approx-nef}
Let $\alpha \in \Nef(\X)$ be any nef class. Then there exists a  sequence of nef Cartier $b$-divisor classes $\alpha_n \in \cNef(\X)$ decreasing to $\alpha$.
\end{restatable}

The crucial point in the statement above is the fact that the approximating sequence is decreasing meaning that it satisfies the condition $\alpha_n \ge \alpha$ in the sense that $\alpha_n - \alpha$ is pseudo-effective for all $n$.
The proof relies on the asymptotic construction of multiplier ideals and is completely analogous to (and in fact simpler than) the
approximation results of relatively nef $b$-divisors (over a smooth point~\cite{MR2426355}, or over a curve~\cite{MR3419957,trivial-valued,MR4018262}).
A similar approximation result is proved in a toroidal context in~\cite[Lemma 5.9]{botero-gil} which does not use multiplier ideals techniques.

After we released a first version of this paper, and using the same technics based on multiplier ideal sheaves, S. Boucksom and M. Jonsson have been able
to prove an approximation result for psh functions on the Berkovich analytification of projective varieties defined over a trivially valued field.
As explained in~\cite[\S 10.3]{global-BJ}, their far-reaching result implies our theorem.

\begin{center} $\diamond$ \end{center}

In order to intersect nef $b$-divisors, it is necessary to discuss an appropriate notion of positivity for
classes of arbitrary codimension.

We first define a $b$-class of codimension $k$ as a collection of numerical classes $\alpha_{X'}\in\num^k(X')$ 
for any proper birational model $X'\to X$ that are invariant under push-forward morphisms. Cartier $b$-classes
are defined analogously as in codimension $1$.  We obtain a locally convex topological space $\wnum^k(\X)$ containing $\cnum^k(\X)$
as a dense subset and we have a natural perfect intersection pairing $\wnum^k(\X) \times \cnum^{d-k}(\X) \to \R$.

Several notions of positivity
have been explored in $\num^k(X')$, see~\cite{MR3592463}. Among those, we shall use the notion of basepoint-free (BPF) classes
which seems appropriate
in several contexts, in particular for applications to dynamical systems, see~\cite{dang_normal,spectral}.

By definition, a numerical class $\num^k(X')$ is BPF if it can be approximated by images under flat morphisms of 
complete intersection ample classes. 
As above, this notion leads to the definition of codimension $k$ BPF numerical $b$-classes.

Note that a class in $\num^1(X')$ is BPF if and only if it is nef, so that  $\BPF^1(\X)=\Nef(\X)$.
We further prove in \S\ref{sec:movable} that a $b$-numerical class $\alpha\in\wnum^k(\X)$ is BPF iff  for all 
models $X'$ the class $\alpha_{X'}$ is movable  in the sense of~\cite[\S 3]{fulger_lehmann_zariski}, see Theorem~\ref{thm:interpret-BPF}. These results suggest
the following duality conjecture that is proved by B.~Lehmann~\cite[Theorem~1.5]{movable-class} when $k=1$, and that we extend to the case $k=d-1$ (Proposition~\ref{prop:dual-BPF}).
\begin{conj}\label{conj:dual-BPF}
~
\begin{itemize}
\item
A $b$-numerical class $\alpha\in\wnum^k(\X)$ is psef
iff for any  $b$-numerical class $\beta\in\cBPF^{d-k}(\X)$ 
 we have $(\alpha\cdot\beta)\ge0$.
\item
A $b$-numerical class $\alpha\in\wnum^k(\X)$ belongs to $\BPF^k(\X)$ iff for any psef Cartier $b$-numerical class $\beta\in\cnum^{d-k}(\X)$,  we have $(\alpha\cdot\beta)\ge0$.
\end{itemize}
\end{conj}

With these notions at hand and using Theorem~\ref{thm:approx-nef}, one can follow the construction of positive intersection classes developed in~\cite{boucksom_favre_jonsson_volume,boucksom_demailly_paun_peternell},
and define a generalized intersection product of nef $b$-divisor classes which takes the following form.

Given any collection of nef classes $\{\alpha_i\}_{1\le i \le k}\in \Nef(\X)$, we define
$(\alpha_1\cdot\ldots\cdot\alpha_k)\in \BPF^k(\X)$ as the infimum $(\beta_1\cdot\ldots\cdot\beta_k)$
over all  nef Cartier $b$-classes $\beta_i\ge \alpha_i$. That this intersection product is multilinear and extends
the natural one on Cartier classes is a consequence of Theorem~\ref{thm:approx-nef}.

Of crucial importance to applications is the fact that the intersection product satisfies 
various forms of Hodge index theorem. In particular, we shall prove:
\begin{restatable}{thmint}{hodgeindex}\label{thm:Hodge-index}
Any two $b$-divisor classes $\alpha, \beta\in\Nef(\X)$ such that $\alpha\cdot \beta=0$ are proportional.
\end{restatable}
We also obtain a far-reaching generalization of~\cite[Corollary~3.5]{MR2066940} for $b$-divisors, see Theorem~\ref{thm:DS} below. 
We hope that this result will find applications to the study of birational group actions on projective varieties.

\begin{center} $\diamond$ \end{center}

Let $\om$ be any Cartier $b$-divisor class determined by an ample class in $X$.
For any $k\ge 1$, let us call big any class $\alpha \in \wnum^{k}(\X)$ such that $\alpha\ge c\om^k$ for some $c>0$.
It follows from Siu's inequalities that for any two big BPF classes $\alpha, \beta\in\BPF^{k}(\X)$ there exists a constant $C>1$
such that $C^{-1} \alpha \le \beta \le C \alpha$ so that the notion of bigness does not depend on the choice of $\om$.

We prove that any big BPF class of dimension $1$ can be realized as a $(d-1)$-th power of a nef $b$-divisor.
Our  precise result reads as follows. 
\begin{restatable}{thmint}{MAmp}\label{thm:MA}

The map $\alpha \mapsto \alpha^{d-1}$ induces a bijection from the
convex cone of big $b$-divisor classes $\alpha\in\Nef(\X)$  onto the convex cone
of big $b$-numerical curve classes in $\BPF^{d-1}(\X)$.
\end{restatable}

Our proof is variational in nature, and follows the same line of arguments as in~\cite{boucksom_favre_jonsson_monge_ampere} (whose proof is
itself inspired by~\cite{alexandrov} and~\cite{berman_boucksom_guedj_zeriahi}). A key point in the proof is the differentiability of the volume function, see~\cite{boucksom_favre_jonsson_volume,MR2571958,MR3981985}. Over a toric variety, a big curve (torus-invariant) class can be identified to 
a positive measure on a sphere, and our statement reduces in that case to Minkowski's theorem which states the existence of a convex body with a prescribed surface area measure~\cite[Theorem~8.2.2]{schneider_convex}.

Recently several decompositions of sufficiently positive curve classes on projective varieties have been introduced~\cite{MATH06656980,fulger_lehmann_zariski,lehmann_xiao_positivity}
which generalize the Zariski decomposition of psef divisors. Theorem~\ref{thm:MA} gives a new perspective on these decompositions as we now explain.

First Lehmann and Xiao~\cite{MATH06656980} have
shown that for any big curve class $\alpha \in \num^{d-1}(X)$ there exists a unique nef class $\LX(\alpha) \in \num^1(X)$
such that 
\[
\alpha \ge \LX(\alpha)^{d-1} \text{ and } \LX(\alpha) \cdot \left(\alpha - \LX(\alpha)^{d-1}\right) =0~.
\]
We prove (Theorem~\ref{thm:relate-to-LX} ) that for any big BPF class $\alpha$ 
the family of nef Cartier classes $\LX(\alpha_{X'})$ where $X'$ ranges over all models over $X$ \emph{converges} to the unique nef class $\gamma$
solving $\gamma^{d-1} = \alpha$. 
As a consequence of~\cite[Theorem~5.29]{MATH06656980},
we also obtain a characterization of big curve classes in terms of the  functional $\hvol$
that was introduced by Xiao in~\cite{zbMATH06798253} (Theorem~\ref{thm:char-of-bigness}).

We then discuss a second decomposition based on the functional $\fM$ defined on the cone of movable curve classes and introduced by Xiao.
We prove (Theorem~\ref{thm:LX-Mfnt}) yet
another convergence result $\Mv(\alpha_{X'})\to \gamma$ where $\Mv(\alpha_{X'})\in\num^1(X')$ is the movable divisor class 
computing $\fM(\alpha_{X'})$ given by~\cite[Theorem~3.14]{lehmann_xiao_positivity}.

\begin{center} $\diamond$ \end{center}

The intersection theory of $b$-divisors and the techniques  presented in this paper lead to a better understanding of the structure of the space of $b$-classes. 
The most natural norm on $\cnum^k(\X)$ is arguably defined by $\normom{\alpha}:= \inf \{ C>0, \, -C \om^k \le \alpha \le C \alpha\}$, and we may 
thus look at the Banach space  $\num^k_\om(\X)$ obtained as the completion of $\cnum^k(\X)$ with respect to this norm.

In a previous paper \cite{spectral}, we also introduced other Banach spaces  $\numbpf{k}(\X), \numbpfs{k}(\X) \subset \wnum^{k}(\X)$ for $0\le k\le d$ and  $\nums(\X)  \subset  \wnum^{1}(\X)$
which played a crucial role in our approach to analyze the degree growth of rational self-maps. 

First $\numbpf{k}(\X)$ is obtained as the completion of the space of Cartier $b$-divisor classes for the norm induced by 
the BPF cone. The space $\numbpfs{k}(\X)$ is similarly defined as the completion of  $\cnum^{k}(\X)$ for the dual norm induced by the BPF cone. 

Let us review briefly the definition of $\nums(\X)$ refereeing to \S\ref{sec:3Banach} below for more details.
When $X$ is a surface, then $\nums(\X)$ is the Picard-Manin space introduced by S. Cantat~\cite{cantat_bir_surfaces} and~\cite{boucksom_favre_jonsson_deggrowth}. 
It is a Hilbert space which is equipped with a natural intersection form  $\langle \cdot, \cdot \rangle$ of Minkowski's type for which Hodge index theorem extends naturally.
It contains $\numbpf{1}(\X)$ and is contained in $\numbpfs{1}(\X)$, and its importance stems from the fact that birational surface maps 
induce isometries for the pairing $\langle \cdot, \cdot \rangle$.

In higher dimension, $\nums(\X)$ is defined as the completion of $\cnum^{1}(\X)$ with respect to a family of semi-norms obtained by restricting classes to 
sufficiently general surfaces in birational models of $X$. In~\cite{spectral}, we proved that a version of the Hodge index theorem still holds in this Banach space.

\smallskip

We use Theorem~\ref{thm:approx-nef} above and a version of Diskant inequalities for $b$-divisors extending~\cite[Theorem~F]{boucksom_favre_jonsson_volume} to prove the following inclusions.

\begin{thmint} \label{thmint_injections}  One has the following sequence of continuous injections:
\begin{equation}\label{eq:inclusion}
\numbpf{1}(\X) \hookrightarrow \Vect(\Nef(\X)) \hookrightarrow \nums(\X) \hookrightarrow \num^1_\om(\X) =\numbpfs{1}(\X)\hookrightarrow \wnum^1(\X). 
\end{equation}
\end{thmint}

The first and last injections were proved in \cite[Proposition 2.10]{spectral} whereas the equality $\num^1_\om(\X) =\numbpfs{1}(\X)$ and the two injections $\Vect(\Nef(\X)) \hookrightarrow \nums(\X) \hookrightarrow \numbpfs{1}(\X)$ are new. 
\begin{conj}
For any projective manifold $X$ of dimension at least $2$, all inclusions in~\eqref{eq:inclusion} are strict.
\end{conj}

The injection $\nums(\X) \hookrightarrow  \numbpfs{1}(\X)$ can be viewed as an analog of the Sobolev injection of $W^{1,2}([0,1])$ into $\mathcal{C}^0([0,1])$.
This analogy takes its root in the context of toric geometry, and we plan to discuss these links in detail in a forthcoming work.

\begin{center} $\diamond$ \end{center}
Let us discuss some of the restrictions of our approach. As in our companion paper \cite{spectral}, we suppose $K$ to be of characteristic $0$.
We use the assumption on the characteristic for the existence of smooth models
and for the uniform generation of multiplier ideals which eventually relies on Kodaira vanishing theorem. We expect though that our results extend verbatim to positive
characteristics with multiplier ideals being replaced by test ideals as in~\cite{MR4018262}. A first step in that direction has been taken by S. D. Cutkosky who proved Diskant 
inequalities for envelopes for big movable divisors, see equation (30) in the proof of Proposition 5.5 in~\cite{cutkosky-preprint}. In the same article, a characterization of the equality case in Minkowski's inequalities is also obtained in arbitrary characteristic. 

\medskip

 It is very tempting to extend our results to compact K\"ahler manifolds. In this case, the natural objects to consider
are projective limits of Dolbeault cohomology spaces $H^{k,k}_\R(X')$ where $\mu\colon X'\to X$ ranges over all 
bimeromorphic proper holomorphic maps.
Observe that one obtains a space which is in general much bigger than $\wnum^k(\X)$ even when $X$ is projective as one needs to deal with potentially transcendental classes.

The notion of volume and the Zariski decomposition of big $(1,1)$-cohomological classes (which in the projective case play important roles in our approach) 
have been defined by Boucksom~\cite{MR2050205} and further characterized by DiNezza-Floris-Trapani~\cite{MR3752530} (see also \cite{nakayama,darvas_chinh_eleonora,tosatti}). 
Even so, we face a  serious obstacle in extending our results
since the differentiability of volume function is not known for arbitrary K\"ahler manifolds. This property is equivalent to the so-called
transcendental Morse inequalities (that we call Siu's inequalities here) by an argument of Xiao, see also~\cite[Appendix~A]{MR3981985}.
Much progress have been made recently towards the proof of these inequalities (see~\cite{MR3449182,MR3541848,MR3451400,MR3441529,MR3752527})
but they remain open in general. Note that the duality between the movable curves and pseudo-effective classes is also very much related to these problems,
and is not known except in the projective case (see~\cite{MR3981985}).

\medskip

It would be interesting to investigate analogs of Theorem~\ref{thm:approx-nef} for a class $\alpha\in\BPF^k(\X)$ with $k\ge2$.
The work of Dinh and Sibony~\cite{MR2119243} on the regularization of positive closed currents 
suggests that 
one might expect the existence of a sequence $\alpha_n^\pm \in \cBPF^k(\X)$ such that  $\alpha = \lim_n (\alpha_n^+ - \alpha_n^-)$, and 
$\sup_n (\alpha^\pm_n\cdot\om^{d-k})\le C\,  (\alpha \cdot\om^{d-k})$ for some universal $C>0$.

One might also want to generalize Theorem~\ref{thm:MA} and try to solve the equation
$(\alpha^k\cdot\om^{d-k-1}) =\theta$ where $\om\in\Nef(\X)$ and $\theta\in\BPF^{d-1}(\X)$ are fixed and $\alpha$ lie in an appropriate subspace of numerical divisor $b$-classes.
This would be analogous to solving mixed Monge-Ampere equations in the complex domain (see~\cite{MR3219505,MR3436233,subsolution,subsolution2} and the references therein
for recent developments on this problem).

\medskip

Finally one can look at relative situations where $X$ is a (non-necessarily Noetherian) scheme, $Z$ is a strict subscheme of $X$, and 
consider all proper morphisms $X' \to X$ that are isomorphisms over $X\setminus Z$. 
A case of interest arises when $X$ is a flat projective scheme over $\spec(K^{\circ})$ where
$K^\circ$ is the ring of integers of a complete non-Archimedean metrized field. Then Theorem~\ref{thm:approx-nef} follows
from the so-called continuity of envelopes, a key problem in the development of non-Archimedean pluripotential theory.
We refer to~\cite{MR3419957,trivial-valued,global-BJ,MR4018262} for a detailed discussion of this important problem.

\subsection*{Acknowledgements}
We would like to thank S\'ebastien Boucksom and  Mattias Jonsson
for their comments and for letting us include in  \S\ref{sec:not-Cartier} an example of a BPF Cartier curve class whose root is not Cartier.
We also thank Ana Maria Botero and Jian Xiao for their numerous remarks on a first version of this paper, and Brian Lehmann for sharing his insights on
the decomposition of numerical classes.
We express our gratitude to the anonymous referee for his/her careful reading and his/her thoughtful suggestions.


\section{Basics on $b$-classes}

In this section, we review briefly some notions on $b$-classes following the discussion in~\cite{spectral}, see also~\cite{boucksom_favre_jonsson_deggrowth,MR2426355,MR3010168}.
We fix a smooth projective variety $X$ of dimension $d$ defined over a field $\K$ of characteristic $0$.

\subsection{Numerical classes of cycles}\label{sec:numerical cycles}
Let $Z^k(X)$ be the $\R$-vector space freely generated by irreducible subvarieties of pure codimension $k$ in $X$. 
Given any two cycles $\alpha,\beta$ of complementary dimension in  $X$, 
we denote by $(\alpha \cdot \beta )\in \R$ their intersection number as defined in \cite{fulton}. 

The numerical space of cycles of codimension $k$, denoted $\num^k(X)$ is defined as the quotient of $Z^k(X)$ by the 
vector space of cycles $z$ such that $(\alpha \cdot z) = 0$ for all cycle $\alpha$ of dimension $k$. It is a finite dimensional $\R$-vector space, and the pairing 
$\num^k(X) \times \num^{d-k}(X) \to \mathbb{R}$ is perfect.

The space $\num^1(X)$ is the tensor product of the Neron-Severi group of $X$ with $\R$. 
Intersection products of $k$ divisors define numerical cycles of codimension $k$ but these classes do not span $\num^k(X)$ in general.

\medskip

A class in $\num^k(X)$  is called pseudo-effective, if it belongs to the closure of the convex cone spanned by effective cycles. 
When $\alpha\in \num^k(X)$ is pseudo-effective, we write $\alpha\ge0$.
The set of pseudo-effective classes forms a closed salient convex cone $\psef^k(X)$ inside $\num^k(X)$ with non-empty interior. 

\medskip

A class lying in the interior of the cone of pseudo-effective classes in $\num^1(X)$ is said to be big. An ample class $\om$ is big. 
A class $\alpha\in\num^1(X)$ is big iff one can find $\epsilon>0$ such that $\alpha \ge \epsilon \om$.

A class $\alpha \in \num^1(X)$ is nef if its intersection with any pseudo-effective curve class is non-negative.
The nef cone $\nef^1(X)$ is the interior of the cone or real ample classes. A nef class $\alpha$ is big iff $(\alpha^d)>0$.

The negativity lemma, see e.g~\cite[Lemma 4.16]{fujino} states that for any proper birational  morphism $\pi\colon X \to Y$ and any nef class $\alpha\in\nef^1(X)$, we have 
$\pi^* (\pi_*\alpha) \ge \alpha$.

\medskip

 A class $\alpha\in\num^k(X)$ with  $k\ge2$ is called strongly basepoint free if it is the pushforward under a flat proper morphism of relative dimension $e$ 
 of the intersection of $k+e$ ample divisors\footnote{This definition is equivalent of the original one given in~\cite{MR3592463}, see~\cite[Corollary 3.3.4]{dang_normal}}. 
The closure of the cone generated by strongly basepoint free classes is called the basepoint free cone of codimension $k$. We denote it by 
$\BPF^k(X)$. Any BPF class is both pseudo-effective and nef, and the pull-back of BPF classes by a proper morphism remains BPF.

\medskip

 A class $\alpha\in\num^k(X)$ is big iff  one can find $\epsilon>0$ such that $\alpha \ge \epsilon \om^k$ for some ample class $\om$.
 A class is bif iff it belongs to the interior of the pseudoeffective cone, see~\cite[Lemma~2.12]{MR3592463}.
  
\medskip

The image under a proper birational map of a BPF class is not BPF in general. This motivates the introduction of the notion of "movable" classes. 
We discuss here only the case of dimension and codimension $1$, leaving the general case to \S\ref{sec:movable}.

 We say that an integral class $\alpha = c_1(L) \in \num^{1}(X)$ is strongly movable if the base locus of a multiple of the line bundle $L\to X$ has codimension $\ge2$. 
 A class  $\alpha \in \num^{1}(X)$ is movable if it lies in the closure of the cone generated by strongly movable integral classes.

The movable cone in $ \num^{d-1}(X)$ is the smallest convex cone which contains
all classes $\alpha \in \num^{d-1}(X)$ for which
 there exists a birational proper morphism $\pi\colon X \to Y$, and real ample  classes $\beta_1, \cdots,  \beta_{d-1} \in \num^{d-1}(Y)$ such that $\alpha= \pi_*( \beta_1\cdot \ldots \cdot \beta_{d-1})$.
 By \cite[Corollary 2.5]{boucksom_demailly_paun_peternell},  a class $\alpha\in\num^{d-1}(X)$ is movable iff $(\alpha \cdot \gamma)\ge 0$ for all nef divisor classes $\gamma\in\num^{1}(X)$.

 \medskip

We now present a key estimate, called Siu's inequalities, which allows one to compare a BPF class with a complete intersection class. We refer to~\cite[Proposition~3.4.6]{dang_normal} for a proof.

\begin{prop} \label{prop_Siu}
There exists a constant $C_d>0$ such that for all classes $\alpha \in \BPF^k(X)$ and all big and nef divisors $\beta$ on $X$, one has: 
\begin{equation}
 \alpha \leq C_d \dfrac{(\alpha \cdot \beta^{d-k})}{(\beta^d)} \beta^k.
 \end{equation} 
\end{prop}


\subsection{Weil and Cartier numerical $b$-classes} \label{section_bclass}

A model  over $X$ is a projective birational morphism $\pi \colon X' \to X$  from a smooth projective variety $X'$. 
Any such model is obtained as the blow-up of some ideal sheaf on $X$ so that the category $\cM_X$ of all models over $X$ 
is naturally a poset for which $X' \ge X''$ iff the canonical map $X' \dashrightarrow X''$ is a morphism.
When $X'\ge X''$ we say that $X'$ dominates $X''$.
The poset $\cM_X$ is inductive in the sense that any two models are simultaneously dominated by a third one.
The projective limit over the inductive system defined by $\cM_X$ where each model is endowed with the Zariski topology is a quasi-compact topological space called the Riemann-Zariski space of $X$ that we denote by $\X$. We refer~\cite{vaquie} for an interpretation of this space in terms of valuations.

\medskip

A Weil numerical $b$-class $\alpha$ of codimension $k$ is a map  which assigns to any smooth model $X'\in \cM_X$ a numerical class $\alpha_{X'}\in \num^k(X')$,
such that $\pi_*(\alpha_{X'})=\alpha_{X''}$ for any pair of smooth models $X'\ge X''$ with $\pi= (\pi'')^{-1}\circ \pi'$. In the rest of the paper,  we will also refer to codimension one $b$-classes as  $b$-divisors. 
The class $\alpha_{X'}$ is called the incarnation of $\alpha$ in $X'$. 

A Cartier $b$-class is a numerical $b$-class $\alpha$ for which one can find a  model $X''$
such that $\alpha_{X'} = \pi^*(\alpha_{X''})$ for any  model $X'\ge X''$. When it is the case, we say that 
$\alpha$ is determined in $X''$. Conversely for any class $\alpha\in\num^k(X')$, we let $[\alpha]$ be the 
Cartier $b$-class determined by $\alpha$ in $X'$. 

\smallskip

The space of Weil numerical $b$-classes is an infinite dimensional real vector space, which we denote by $\wnum^k(\X)$.
It contains the set of Cartier numerical $b$-classes $\cnum^k(\X)$ as a subspace, and for each smooth model $X'$ the map
$\alpha \mapsto [\alpha]$ induces an injective linear map $\num^k(X')\to\cnum^k(\X)$.

\smallskip

The space  $\wnum^k(\X)$ is a locally convex topological vector space when endowed with its natural product topology
such that every bounded set is relatively compact. 
Note that when $K$ is countable then it is separable and metrizable, but these properties do not hold when $K$ is uncountable.

In this topology, a sequence $\alpha_n \in \wnum^k(\X)$ converges to $\beta \in \wnum^k(\X)$ iff
for any smooth model $X'$ we have $(\alpha_n)_{X'} \to \beta_{X'}$ in $\num^k(X')$.
Since for any $\alpha\in\wnum^k(\X)$ the sequence of Cartier $b$-classes $[\alpha_{X'}]$ converges to $\alpha$,
the space $\cnum^k(\X)$ is dense in $\wnum^k(\X)$.

Pick $\alpha\in \cnum^k(\X)$ and $\beta\in \wnum^l(\X)$. 
Suppose $\alpha$ is determined in a model $X_0$. For any other smooth model $X'\ge X_0$, we set 
$(\alpha \cdot \beta)_{X'} := \alpha_{X'} \cdot \beta_{X'} \in \num^{k+l}(X')$. 
By the projection formula, $\pi_* (\alpha \cdot \beta)_{X''} = (\alpha \cdot \beta)_{X'}$ for any smooth model $X''\ge X'$
so that we may define the class $\alpha \cdot \beta \in \wnum^{k+l}(\X)$ as the unique Weil numerical $b$-class
whose incarnation in any smooth model dominating $X_0$ is equal to $(\alpha \cdot \beta)_{X'}$. 

The ring $\cnum^\bullet(\X) = \oplus \cnum^k(\X)$ is a graded
ring over which $\wnum^\bullet(\X) = \oplus \wnum^k(\X)$ is a graded module, and 
the pairing $\cnum^\bullet(\X) \times \wnum^\bullet(\X)\to \R$ is perfect.

\subsection{Positive cones of $b$-numerical classes}

We discuss here some positivity notions on $b$-classes that are inherited from those defined on particular models.

A class $\alpha\in\wnum^k(\X)$ is said to be pseudo-effective (and we write $\alpha\ge0$) when $\alpha_{X'}\ge0$ for any model $X'$.
Note that the pseudo-effectivity is only preserved by push-forward, and not by pull-back, so that it may happen that a Cartier $b$-class
determined by a pseudo-effective class in a smooth model $X$ is not pseudo-effective in $\wnum^k(\X)$. 
However a Cartier $b$-divisor $\alpha \in \cnum^1(\X)$ is pseudo-effective iff its incarnation in one (or all) of its determination on a smooth model is pseudo-effective.

\smallskip

The notion of base-point free $b$-class is defined as follows. 
We let $\cBPF^k(\X)$ be the convex cone in $\cnum^k(\X)$ generated by Cartier numerical $b$-classes
$[\alpha]$ with $\alpha\in \BPF(X')$ for some  model $X'$.
Since BPF classes are stable by pull-back, a class $\alpha\in\cnum^k(\X)$ is BPF iff it is BPF in one (or any)
of its determination. 

\begin{defi} 
The cone $\BPF^k(\X)$ is the (weak) closure in $\wnum^k(\X)$ of the cone $\cBPF^k(\X)$.
\end{defi}
Since the cone $\BPF^1(X')$ coincides with the nef cone for all model $X'$, 
we write $\Nef(\X)=\BPF^1(\X)$, and $\cNef(\X)=\cBPF^1(\X)$.

\begin{lem}
A class $\alpha\in \wnum^k(\X)$ is BPF iff there exists a net of classes $\alpha_i \in \cBPF^k(\X)$ such that 
 $(\alpha_i)_{X'} \to \alpha_{X'}$ for all model $X'$.
 \end{lem}
\begin{proof}
Pick any class $\alpha\in\wnum^k(\X)$. For any model $X'$, and for any open subset $U$ of $\num^k(X')$, let
$[U,X'] = \{ \beta \in \wnum^k (\X), \beta_{X'} \in U\}$. Then the collection $I_\alpha$ of all sets containing $\alpha$ and of the form
$[U,X']$ forms a basis of neighborhood of $\alpha$. It is also an inductive set for the order relation
$[U_0,X_0] \le [U_1,X_1]$ iff $X_0 \le X_1$ and the preimage of $U_0$ by the pullback morphism $\num^k(X_0)\to\num^k(X_1)$ contains
$U_1$. 

Suppose that $\alpha$ is BPF. Then for each $i=[U,X']\in I_\alpha$ we may choose a Cartier BPF $b$-class $\alpha_i$ 
in $[U,X']$. By construction, we have  $(\alpha_i)_{X'} \to \alpha_{X'}$ for all model $X'$.
The converse statement is clear.
\end{proof}

\begin{lem}\label{lem:negativity}
For any class $\alpha\in \nef^1(\X)$, and for any model $X''\ge X'$ we have
\[\alpha \le [\alpha_{X''}]\le [\alpha_{X'}]~.\]
 \end{lem}
\begin{proof}
This is a rephrasing of the negativity lemma alluded to above.
\end{proof}
\subsection{The Banach spaces $\numbpf{k}(\X) , \numbpfs{k}(\X)$ and $\nums(\X)$}\label{sec:3Banach}
Let $\om$ be any fixed big and nef Cartier $b$-divisor class.
For any class $\alpha \in \cnum^k(\X)$, we define:
\begin{align}
\normbpf{\alpha} &:= \inf_{\substack{\alpha = \alpha_+ - \alpha_- \\
\alpha_{\pm}\in \BPF^{k}(\X) }} (\alpha_+ \cdot \om^{d-k}) + (\alpha_- \cdot \om^{d-k}),
\\
\normbpfs{\alpha} &:= \sup_{\substack{\gamma \in \cBPF^{d-k}(\X)\\
(\omega^k \cdot \gamma) =1 
}} |(\alpha \cdot \gamma)|,
\end{align}
and when $\alpha \in \cnum^1(\X)$, we set:
\begin{equation} \label{def_norm_l2}
\normsigma{\alpha} :=\sup_{ \substack{\gamma \in \cBPF^{d-2}(\X) \\
(\gamma \cdot \om^2) = 1}}  \left( 2 (\alpha \cdot  \om \cdot \gamma )^2  -  ( \alpha^2 \cdot \gamma) \right)^{1/2}.
\end{equation}
We showed in~\cite[\S 2--3]{spectral} that these functions induce norms on $\cnum^k(\X)$ and $\cnum^1(\X)$ respectively.
\smallskip 

One defines the spaces $\numbpf{k}(\X)$, $\numbpfs{k}(\X)$ as the completions of $\cnum^k(\X)$ with respect to $\normbpf{\cdot}$ and  $\normbpfs{\cdot}$ respectively. 
The space $\nums(\X)$ is similarly obtained as the completion of $\cnum^1(\X)$ with respect to the norm $\normsigma{\cdot}$.

It was proved in~\cite[Theorem 3.3]{spectral}, that $\nums(\X)$  is a vector subspace of $\wnum^1(\X)$ which does not depend on $\om$. 
We also obtained the following statement (see \cite[Theorem 4.11]{spectral}). 

\begin{thm} \label{thm_intersection_nums} The intersection product $\cnum^1(X)\times \cnum^1(\X) \to \cnum^2(\X)$ extends continuously to a symmetric bilinear map: 
\begin{equation}
\nums(\X) \times \nums(\X) \to \numbpfs{2}(\X).
\end{equation}
\end{thm}

In \cite[Proposition 2.10 and Theorem 3.16]{spectral}, we proved the following continuous injections:
\begin{equation} 
\label{eq_known_inclusions}
\begin{array}{cl}
\numbpf{k}(\X) & \hookrightarrow  \numbpfs{k}(\X)  \hookrightarrow \wnum^k(\X), 
\\
\numbpf{1}(\X) & \hookrightarrow \nums(\X) \hookrightarrow  \wnum^{1}(\X). 
\end{array}
\end{equation}

Let us record the following useful lemma.
\begin{lem} \label{lem_straightforward} For any $\alpha \in \Nef(\X)$, one has $(\alpha \cdot \om^{d-1}) = \normbpf{\alpha}$.
\end{lem}

\begin{proof}
Fix a class  $\alpha \in \Nef(\X)$, and note that it follows from the definition that
\begin{equation*}
\normbpf{\alpha} \leq (\alpha \cdot \om^{d-1}). 
\end{equation*}
Conversely, write $\alpha = \alpha_+ - \alpha_-$ where $\alpha_\pm \in \Nef(\X)$ are chosen so that $(\alpha_+ \cdot \om^{d-1}) + (\alpha_- \cdot \om^{d-1}) \leq \normbpf{\alpha} + \epsilon $ with $\epsilon >0$. Then $(\alpha \cdot \om^{d-1}) \leq (\alpha_+ \cdot \om^{d-1}) + (\alpha_- \cdot \om^{d-1})\leq \normbpf{\alpha} +\epsilon $, and this implies the reverse inequality.  
\end{proof}

\subsection{The Banach space $\num^k_\om(\X)$}
As in the previous section, $\om$ is a big and nef Cartier $b$-divisor class.
For any class $\alpha\in\cnum^k(\X)$, we set
\[
\normom{\alpha}:= \inf \left\{C\ge 0, \, 
 - C \omega^k \le \alpha \le C \omega^k\right\}~.
\]
Since the pseudo-effective cone is closed in any model, note that the infimum is actually attained, and
the function $\normom{\cdot}$ defines a norm on $\cnum^k(\X)$.
\begin{lem} \label{lem_estnorm}
There exists a positive constant $C_d>0$ depending only on $d$ such that the following hold:
for any class  $\alpha\in\cBPF^k(\X)$,
\begin{equation}\label{eq:estim-normom}
\frac1{(\om^d)} (\alpha\cdot\om^{d-k})
\le
\normom{\alpha}
\le
C_d (\alpha\cdot\om^{d-k})~;
\end{equation}
and for any class $\alpha\in\cnum^k(\X)$,
\[
\normbpfs{\alpha}
\le 
\normom{\alpha}
\le 
C_d\,\normbpf{\alpha}~.
\]
\end{lem}
\begin{proof}
Suppose $\alpha$ is BPF. By Siu's inequalities (Proposition \ref{prop_Siu}), we get $0\le \alpha \le C_d (\alpha\cdot\om^{d-k})\, \om^k$
so that $\normom{\alpha} \le C_d (\alpha\cdot\om^{d-k})$. Conversely,  $\alpha \le C \om^k$
implies $(\alpha\cdot\om^{d-k})\le C (\om^d)$ hence~\eqref{eq:estim-normom} holds.

Pick any Cartier $b$-class $\alpha$, and write 
$\alpha= \alpha_+ - \alpha_-$ with $\alpha_\pm\in \BPF^k(\X)$.
The previous arguments show
\[
\alpha =  \alpha_+ - \alpha_- \le \alpha_+ \le C_d (\alpha_+\cdot\om^{d-k})
\le C_d \normbpf{\alpha}
\]
and similarly $\alpha \ge -  C_d \normbpf{\alpha}$ hence $\normom{\alpha} \le C_d \normbpf{\alpha}$.

Suppose that $- C \om \le \alpha \le C \om$ for some positive constant $C>0$, and pick any class $\gamma\in\BPF^{d-1}(\X)$.
We obtain $|(\alpha\cdot \gamma)|\le C (\om\cdot \gamma)$ hence $\normbpfs{\alpha}\le C$.
\end{proof}

Denote by $\num^k_\om(\X)$ the completion of $\cnum^k(\X)$ for the norm $\normom{\cdot}$.
The previous estimate proves that we have
the following continuous injections:
\begin{equation} 
\label{eq_known_inclusions2}
\begin{array}{cl}
\numbpf{k}(\X) & \hookrightarrow \num^k_\om(\X)  \hookrightarrow  \numbpfs{k}(\X).
\end{array}
\end{equation}

We observe that both in dimension and codimension $1$, the norm $\normom{\cdot}$ is identical to the dual BPF norm.
\begin{prop}\label{prop:BDPP}
For every class $\alpha \in \cnum^{1}(\X)\cup \cnum^{d-1}(\X)$, we have
\[
\normbpfs{\alpha}=\normom{\alpha}~.
\]
In particular, $\num^1_\om(\X)  =  \numbpfs{1}(\X)$ and $ \num^{d-1}_\om(\X)  =  \numbpfs{d-1}(\X)$.
\end{prop}

\begin{proof}
Suppose $\alpha$ is a Cartier class determined in some smooth model $X'$.
We only need to show that $\normom{\alpha}\le \normbpfs{\alpha}$. 

Let us first treat the case $\alpha \in \cnum^{d-1}(\X)$. For any 
nef class $\gamma\in \nef(X')$ such that $(\gamma \cdot \om^{d-1})=1$ we have 
\[
\left((\alpha  + \normbpfs{\alpha}\, \om^{d-1})\cdot[\gamma]\right) 
\ge 0
\]
hence the class $(\alpha  + \normbpfs{\alpha}\, \om^{d-1})_{X'}$ is pseudo-effective by duality.
Similarly  we obtain $(-\alpha +  \normbpfs{\alpha}\, \om^{d-1})_{X'}\ge0$ hence
$\normom{\alpha}\le \normbpfs{\alpha}$ as required.

\medskip

Let us now treat the case where $\alpha $ is a $b$-divisor.
Recall from~\cite{boucksom_demailly_paun_peternell} that a class $\alpha_0\in\num^1(X)$ is pseudo-effective if and only if $(\alpha_0\cdot \gamma_0)\ge0$ for any movable curve class $\gamma_0\in\num^{d-1}(X)$.

Suppose first that $\alpha$ is a Cartier $b$-divisor determined in some smooth model $X'$. For any movable class $\gamma_0\in\num^{d-1}(X')$, there exists a class $\gamma\in \cBPF^{d-1}(\X)$ such that $\gamma_{X'} = \gamma_0$.
It follows that
\[
\left((\alpha  + \normbpfs{\alpha}\, \om)\cdot\gamma_0\right) 
= 
\left((\alpha + \normbpfs{\alpha}\, \om)\cdot\gamma\right) \ge 0
\]
hence $\alpha + \normbpfs{\alpha} \om\ge0$ as required.
\end{proof}


\section{Approximation of nef $b$-divisor classes}

In this section, we prove a version of our approximation result from the introduction in terms of nets. We shall see later in \S\ref{sec:Diskant} 
that we can replace nets by sequences and obtain Theorem~\ref{thm:approx-nef}.

\begin{thm}\label{thm:approx2}
Let $\alpha \in \Nef(\X)$ be any nef class. 
There exists an inductive set $I$ and a non-increasing net of 
nef Cartier $b$-divisor classes $(\alpha_i)_{i\in I} \in \cNef(\X)$ 
 that is converging to $\alpha$.
\end{thm}
 In \S\ref{section_ideal} and~\ref{section_nef_envelope} we first collect some general facts on ideal sheaves, and nef envelopes.
 The proof of Theorem~\ref{thm:approx2} is then given in \S\ref{section_pf_approx}. 
In \S\ref{sec:first-app}, we give a first application of this theorem and prove that a nef divisor $b$-class which is Cartier is determined by a nef class in some model.

\subsection{Coherent ideal sheaves} \label{section_ideal}

To any coherent ideal sheaf $\fA$ on $X$, we attach a Cartier $b$-divisor $[\fA]$ 
as follows. Choose a log-resolution $X'$ of $\fA$ so that one can write
$\mathcal{O}_{X'}(-D) =\fA\cdot \mathcal{O}_{X'}$
for some divisor $D$. Then  we let $[\fA]$ be the Cartier $b$-divisor determined by $[D]$ 
in $X'$.
Note that this definition does not depend on the choice of resolution, and 
since $D$ is effective that $[\cI] \ge0$.

\smallskip

For any pair of ideal sheaves $\fA, \fB$ on $X$, we have
\begin{enumerate}
\item[(i)]  $[\fA]\ge [\fB]$ when $\fA \subset \fB$, and
\item[(ii)] $[\fA\cdot \fB]=[\fA]+[\fB]$.
\end{enumerate}

The notion of multiplier ideal sheaf will play an essential role in the paper. We briefly discuss the basic definitions and properties
that are necessary for our purposes, refereeing the interested reader to~\cite{MR2095472} for an extensive treatment of this notion.

Let $\fA$ be any coherent ideal sheaf on $X$, and let $c>0$
be any positive constant.
If $X' \overset{\pi}{\to} X$ is a resolution of $\fA$ as above, then we let  $\mathcal{J}(\fA^c)=\pi_* (K_{X'/X} -  \left \lfloor{c\cdot D}\right \rfloor
)$
where $K_{X'/X}= \pi^* K_X- K_{X'}$ is the relative canonical sheaf, and $ \left \lfloor{c\cdot D}\right \rfloor$ is the round-down of the divisor $c\cdot D$.
This definition does not depend on the resolution.

A graded sequence of coherent ideal sheaves $\fA_\bullet= (\fA_m)$ is by definition  a sequence 
of coherent ideal sheaves such that $\fA_0 = \cO_X$ and 
$\fA_n\cdot \fA_m\subset \fA_{n+m}$ for all $n,m\ge0$. For any $c>0$,
one can show that $\cJ(\fA_m^{c/m}) \subset \cJ(\fA_n^{c/n})$ whenever $m$ divides $n$, and one defines
the multiplier ideal sheaf $\cJ(\fA_\bullet^c)$ as the unique maximal ideal sheaf containing all $\cJ(\fA_m^{c/m})$. 
This definition makes sense by Noetherianity. Moreover, we have
\begin{equation}\label{eq:4567}
\fA_m \subset \mathcal{J}(\fA_\bullet^m)
\end{equation} 
for all $m>0$ by~\cite[Theorem~11.1.19]{MR2095472},
and  the following fundamental
sub-additivity property holds (see~\cite[Theorem~11.2.3]{MR2095472}):
\begin{equation}\label{eq:sub-additivity}
\mathcal{J}(\fA_\bullet^{c(m+n)}) \subset \mathcal{J}(\fA_\bullet^{cn})\cdot\mathcal{J}(\fA_\bullet^{cm})
\end{equation} 
for any graded sequence of coherent ideal sheaves $\fA_\bullet$.

 Suppose now that $L\to X$ is a big line bundle. For $n$ large enough, the space of global sections $H^0(X,L^{\otimes n})$ is non zero, and
 we may consider the base locus $\mathfrak{b}_n(L)$  which is the coherent ideal sheaf locally defined by the vanishing of local sections of $L^{\otimes n}$.
 The sequence $\mathfrak{b}_\bullet(L):=\{\mathfrak{b}_n(L)\}$ forms a graded sequence of ideal sheaves and Nadel's vanishing theorem implies the
 following theorem, see~\cite[Corollary 11.2.13]{MR2095472}.

\begin{thm}\label{thm:global-gen}
Fix any very ample line bundle $A$ on $X$ such that $K_X + (d+1)A$ is effective.
For any big line bundle $L\to X$, and for any $m \geq 1$, the sheaf 
\begin{equation}
\mathcal{O}_X(K_X + (d+1) A)\otimes L^m \otimes \cJ(\mathfrak{b}^m_\bullet)
\end{equation}
is globally generated.
\end{thm}

\subsection{Nef envelopes} \label{section_nef_envelope}

In this section, we gather some facts on the notion of nef envelopes following~\cite{boucksom_favre_jonsson_volume,MR2931273,MR3010168}. 
These notions were extended to arbitrary fields in~\cite{MR3368254}.

\smallskip

Let us fix any big Cartier $b$-divisor class $\beta\in\cnum^1(\X)$. 
Consider the set $\mathcal{D}(\beta)$ of classes $\gamma\in\cNef(\X)_\Q$ such that $\gamma\le \beta$.
By~\cite[Lemma~2.6]{boucksom_favre_jonsson_volume} or~\cite[Lemma~3.4]{MR3010168}, 
for any $\gamma_1, \gamma_2\in\mathcal{D}(\beta)$ one can find a third element $\gamma_3\in\mathcal{D}(\beta)$
 such that $\gamma_3 \ge \gamma_1$ and $\gamma_3\ge\gamma_2$. 
 For any model $X'$ over $X$, we may thus define  $P(\beta)_{X'}$ as the least upper bound of 
 the set $\{ \gamma_{X'} \in\num^1(X'), \, \gamma\in \mathcal{D}(\beta)\}$
 (see~\cite[Lemma~2.7]{boucksom_favre_jonsson_volume}). The collection of classes $\{P(\beta)_{X'}\}$ defines a nef $b$-divisor class (called the nef envelope of $\beta$)
 such that for any $\gamma\in\cNef(\X)$ satisfying $\gamma\le \beta$, we have $\gamma\le P(\beta)$.

\begin{lem}\label{lem:approx}
For any $\gamma\in\Nef(\X)$ such that  $\gamma\le \beta$, we have $\gamma\le P(\beta)$;
and the nef envelope is the least nef class satisfying this property.
\end{lem}
\begin{rem}
The nef envelope  coincides with the positive product intersection $\langle \beta \rangle$ of~\cite[\S 2.2]{boucksom_favre_jonsson_volume}.
\label{p:nef-envelope}
It is also the nef part of the Zariski decomposition defined in~\cite{MR3010168}.
\end{rem}

\begin{proof}
Suppose $\beta\ge \gamma\in\Nef(\X)$, and $\beta$ is determined in $X'$. Then $\beta_{X'}\ge \gamma_{X'}$ and there exists a sequence $\gamma_n\in\cNef(\X)$
such that $\gamma_{n,X'} \to \gamma_{X'}$. Since $\beta_{X'}$ is big, we have $\gamma_{n,X'}\le (1+\epsilon)\beta_{X'}$ for any $\epsilon>0$ and for $n \geqslant N_\epsilon$ large enough.
We infer from Lemma~\ref{lem:negativity} $\gamma_n\le [\gamma_{n,X'}] \le (1+\epsilon)\beta$ hence $\gamma_n\le (1+\epsilon) P(\beta)$ for $n \geqslant N_\epsilon$.
We conclude that   $\gamma\le P(\beta)$ by letting first $n\to\infty$ and then $\epsilon\to 0$.

Suppose that $\beta'$ is a nef $b$-divisor class such that $\gamma\le \beta'$ for any 
$\gamma\in\Nef(\X)$ satisfying $\gamma\le \beta$. Applying these inequalities to
Cartier $b$-divisor classes, we get $P(\beta)_{X'} \le \beta'_{X'}$ in each model, which concludes the proof.
\end{proof}

We shall use the following characterization of the nef envelope.
\begin{prop}\label{prop:movable-part}
Let $L\to X$ be any big line bundle. Then
\[
P([c_1(L)]) = \lim_{n\to\infty}\, [c_1(L)] - \frac1n [\mathfrak{b}_n(L)]
\]
where $[\mathfrak{b}_n(L)]$ denotes the base locus ideal of sheaves of the line bundle $L^{\otimes n}$.
\end{prop}
See~\cite[Proposition~3.7]{MR3010168} for the proof.

\subsection{Proof of Theorem~\ref{thm:approx2}} \label{section_pf_approx}
We fix any  Cartier $b$-divisor class $\om$ determined by an ample class in $X$.
We rely on the following proposition whose proof is given below.

\begin{prop} \label{pro:good-approx}
For any nef $b$-divisor class $\alpha$, and any big class $\gamma \in \num^1(X')$ in a model $X'$, 
there exists a nef Cartier $b$-divisor class $\beta$ such that
\[
\alpha \le \beta \text{ and } \beta_{X'} \le \alpha_{X'} + \gamma\]
\end{prop}

Pick any class $\alpha \in \Nef(\X)$. Denote by $I$ the set of all nef Cartier $b$-divisor classes $\beta$ such that 
$\beta \ge \alpha + \epsilon \om$ for some $\epsilon >0$. We endow the set $I$ with a partial order
by declaring that $\beta$ is lower or equal to $\beta'$ iff $\beta' \leqslant \beta$. 

Let us first justify that $I$ is an inductive set. Pick any two nef Cartier $b$-divisor classes $\beta, \beta' \in I$. 
We need to exhibit a class $\beta'' \in I$ such that $\beta''\le\beta$ and $\beta''\le \beta'$. 

We may suppose that both classes are determined by a class in the same model $X'$ and 
that $\beta \ge \alpha + \epsilon \om$ and  $\beta' \ge \alpha + \epsilon \om$ for the same constant $\epsilon>0$. 
By Proposition~\ref{pro:good-approx}, there exists a nef Cartier $b$-divisor class $\beta''$
such that $\alpha + \frac{\epsilon}2 \om \le \beta''$ and $\beta''_{X'} \le \alpha_{X'} + \epsilon \om_{X'}$. 
By Lemma~\ref{lem:negativity}, this implies $\beta'' \le \beta$ and $\beta''\le \beta'$ proving that $I$ is an inductive set. 

In order to make a distinction between nef classes and the poset $I$ parameterizing the approximants of $\alpha$,
let us denote by $\beta_i$ the class associated to the element $i\in I$. 
Pick any model $X'$, and choose
an open neighborhood $U$ of $\alpha_{X'}$ in  $\num^1(X')$. 
One can find 
$\epsilon>0$ such that any class $\gamma\in\num^1(X')$ satisfying 
$\alpha_{X'} \le \gamma \le \alpha_{X'} + 3 \epsilon \om_{X'}$
belongs to $U$. By  Proposition~\ref{pro:good-approx}, there exists a nef Cartier $b$-divisor class
$\beta_i$ such that $\alpha + \epsilon \om \le \beta_i$ and $\beta_{i,X'} \le \alpha_{X'} + 3 \epsilon \om_{X'}$
For any index $j \ge i$, we have
\[
\alpha_{X'}
\le (\beta_j)_{X'} \le(\beta_i)_{X'} \le
 \alpha_{X'} + 3\epsilon \om_{X'}
 \]
 hence $(\beta_j)_{X'}$ belongs to $U$. This proves the net $\beta_i$ is converging to $\alpha$. 
 We conclude by observing that the net is decreasing by construction.

\begin{proof}[Proof of Proposition~\ref{pro:good-approx}]
Since  the big cone is open in $\num^1(X')$, 
one can find a big line bundle $L\to X'$ and an integer $l\in \N$ such that 
$\alpha_{X'} \le \frac1l c_1(L) \le \alpha_{X'} + \frac{\gamma}2$. 
By Proposition~\ref{prop:approx-big envelope} below, there exists a sequence of nef Cartier $b$-divisor classes
$\beta_n$ decreasing to $P(\frac1l c_1(L))$. Set $\beta := \beta_N$ for a sufficiently large integer $N$
so that
$\beta_{X'} \le  \alpha_{X'} + \gamma$. We infer from Lemma~\ref{lem:approx} that 
$\beta \ge P(\frac1l c_1(L)) \ge\alpha$ and the proof is complete.
\end{proof}

\begin{prop} \label{prop:approx-big envelope}
Let $L\to X$ be any big line bundle and set $\alpha := [c_1(L)] \in\cnum^1(\X)$. 
Then there exists a non-increasing sequence of Cartier $b$-divisor classes
$\beta_n\in \cNef(\X)$ converging to the envelope $P(\alpha)$. Moreover, for any $\epsilon>0$ we have $P(\alpha) \le \beta_n\le (1+\epsilon) P(\alpha)$
for all $n$ large enough.
\end{prop}

\begin{proof}[Proof of Proposition \ref{prop:approx-big envelope}]
Replacing $L$ by one of its multiple, we may (and shall) find an effective divisor $D$ such that $L = \cO_X(D)$.
Let $\mathfrak{b}_\bullet:=\{\mathfrak{b}_n(L)\}$ be the graded sequence of ideal sheaves defined by the base loci of $L^{\otimes n}$.
Fix any very ample line bundle $A$ on $X$ such that $K_X + (d+1)A$ is effective. 
By Theorem~\ref{thm:global-gen}, the sheaf 
\[\mathcal{O}_X( K_X+ (d+1) A+ mD) \otimes \cJ(\mathfrak{b}_\bullet^m)\] is globally generated
so that we may find a nef divisor $D_m$ on a log-resolution $X_m$ of the sheaf of ideals $J(\mathfrak{b}_\bullet^m)$
such that 
\[\mathcal{O}_{X_m}( K_X+ (d+1) A+ mD) \otimes  \cJ(\mathfrak{b}_\bullet^m) = \cO_{X_m}(D_m)~.\]
If $\beta'_m$ denotes the nef Cartier $b$-divisor class determined by $D_m$ in $X_m$, then the incarnation of $\beta'_m$ in $X$ is less effective than $ K_X+ (d+1) A+ mD$, hence
\begin{equation}\label{eq:001}
\frac1m \beta'_m \le P\left(\alpha+ \frac{[K_X+ (d+1)A]}m\right) \le P ((1+\epsilon) \alpha)
\end{equation}
for any fixed $\epsilon>0$, and any $m$ large enough.
\smallskip

Now by the sub-additivity of multiplier ideals~\eqref{eq:sub-additivity}, we have $\cJ(\mathfrak{b}_\bullet^{n+m})\subset \cJ(\mathfrak{b}_\bullet^m)\cdot \cJ(\mathfrak{b}_\bullet^n)$, 
hence $-[\cJ(\mathfrak{b}_\bullet^{n+m})]\le -[\cJ(\mathfrak{b}_\bullet^m)]- [\cJ(\mathfrak{b}_\bullet^n)]$ and the sequence 
$-\frac1{2^m} [\cJ(\mathfrak{b}_\bullet^{2^m})]$ is decreasing.
It follows that $\beta_n:= \frac1{2^n} \beta'_{2^n}$ is also decreasing.
By~\eqref{eq:4567}, we have $\mathfrak{b}_m \subset \mathcal{J}(\mathfrak{b}_\bullet^m)$, and we infer
\[
\beta_n
= 
\frac1{2^n}[K_X+ (d+1) A]+ \alpha - \frac1{2^n} [\mathcal{J}(\mathfrak{b}_\bullet^{2^n})]
\ge 
\alpha - \frac1{2^n}[\mathfrak{b}_{2^n}]
~.\]
By Proposition~\ref{prop:movable-part}, we have $\alpha - \frac1{2^m}[\mathfrak{b}_{2^m}]\to P(\alpha)$ and combining this information with~\eqref{eq:001} we conclude that 
$\beta_n \to P(\alpha)$.
\end{proof}

\subsection{Application of the approximation theorem} \label{sec:first-app}
Our next result is an analog of~\cite[Proposition~2.2]{MR2426355},~\cite[Theorem~5.11]{MR3419957},~\cite[Theorem~5.19]{trivial-valued},~\cite[Theorem~1.3]{MR3960125}, and~\cite[Lemma~4.24]{botero-gil}.

\begin{thm} \label{thm:cartier_nef} 
We have \[\cnum^{1}(\X) \cap \Nef(\X) = \cNef(\X).\] 
In other words any Cartier numerical $b$-divisor class lying in the weak closure of nef classes is determined by a nef class in some model.
 \end{thm}

\begin{proof}
Pick any $\alpha \in \cnum^{1}(\X) \cap \Nef(\X)$. Without loss of generality, we may suppose that $\alpha$ is determined in $X$. 
By adding a small real ample class to $X$, we may suppose that $\alpha_X = c_1(L)$ for some big line bundle $L\to X$. 
Since $\alpha$ is nef,  we have $P(\alpha_X)=\alpha$.
Redoing the proof of Proposition~\ref{prop:approx-big envelope}, we find a sequence of elements 
\[\beta_m=\frac1{2^m} \beta'_{2^m} =  \frac1{2^m}[K_X+ (d+1) A]+ \alpha - \frac1{2^m} [\mathcal{J}(\mathfrak{b}_\bullet^{2^m})]\in \cNef(\X)\] which converges weakly to $\alpha$.
This implies $- \frac1{2^m} [\mathcal{J}(\mathfrak{b}_\bullet^{2^m})]$ to tend to $0$ weakly. Since the sequence $- \frac1{2^m} [\mathcal{J}(\mathfrak{b}_\bullet^{2^m})]\le 0$ is decreasing, we get $\cJ(\mathfrak{b}_\bullet^{2^m})=0$ for all $m$, 
and $\beta_{m}$ is thus a nef class determined in $X$. This implies $\alpha_X$ to be nef as required.
\end{proof}

\begin{rem} 
We shall prove an analog  statement  for BPF curve $b$-numerical classes, see Proposition~\ref{pro:cartier_nef_curve} below.
The equality $\cnum^{k}(\X) \cap \BPF^k(\X) = \cBPF^k(\X)$
remains however unclear when $2\le k \le d-2$.
\end{rem}

\section{Intersection of nef $b$-divisor classes}

\subsection{Definition}

\begin{prop-def}\label{prop:exists-key}
Let $\alpha_1, \cdots, \alpha_k$ be any family of nef $b$-divisor classes, and pick $\gamma\in\BPF^l(\X)$.
Then there exists a unique class $\Delta \in\BPF^{k+l}(\X)$ satisfying the following properties.
\begin{enumerate}
\item
For any nef Cartier $b$-divisor classes $\beta_i\ge \alpha_i$, we have
 \begin{equation}\label{eq:inf-nef}
\Delta \le (\beta_1\cdot \beta_2 \cdot \ldots \cdot \beta_k\cdot\gamma)~. 
\end{equation}
\item
For any net of nef Cartier $b$-divisor classes  $ \beta^{(i)}_1, \cdots,  \beta^{(i)}_k$
decreasing to  $\alpha_1, \cdots, \alpha_k$, 
we have 
 \begin{equation}\label{eq:inf-nef2}
( \beta^{(i)}_1\cdot  \beta^{(i)}_2 \cdot \ldots \cdot  \beta^{(i)}_k\cdot\gamma)\downarrow \Delta.
\end{equation}
\end{enumerate}
\end{prop-def}

This intersection product was defined in~\cite[\S 2]{boucksom_favre_jonsson_volume} following the construction in the complex setting of positive intersection products introduced by Boucksom in~\cite[\S 3.2]{boucksom2002cones}, see also \cite{boucksom_demailly_paun_peternell}. It bears a strong analogy with the definition of wedge products of currents $dd^c u_i$ in complex geometry.

\begin{proof}[Proof of Proposition-Definition~\ref{prop:exists-key}]
By Theorem~\ref{thm:approx2} one can find a net of nef Cartier $b$-divisor classes $\alpha^{(i)}_l\downarrow \alpha_l$ for all $l=1, \cdots, k$. 
Since $\gamma$ is BPF, for any $j\ge i$ we have 
\[
(\alpha^{(i)}_1\cdot \alpha^{(i)}_2 \cdot \ldots \cdot \alpha^{(i)}_k \cdot \gamma) \ge (\alpha^{(j)}_1\cdot \alpha^{(i)}_1 \cdot \ldots \cdot \alpha^{(i)}_k \cdot \gamma)
\ge (\alpha^{(j)}_1\cdot \alpha^{(j)}_1 \cdot \ldots \cdot \alpha^{(j)}_k \cdot \gamma)\]
hence the net 
$(\alpha^{(i)}_1\cdot \alpha^{(i)}_2 \cdot \ldots \cdot \alpha^{(i)}_k \cdot \gamma)$ is decreasing.
Let $\Delta$ be its (weak) limit. 

To prove~\eqref{eq:inf-nef}, take nef Cartier $b$-divisor classes $\beta_l\ge \alpha_l$, $l=1, \cdots, k$. Choose a model $X'$ in which all classes $\beta_1, \cdots, \beta_k$ are determined. 
Fix any Cartier $b$-divisor class $\om$ determined in $X$ by an ample class and a positive constant $C>0$ such that $\beta_l \le C \om$ for all $l =1, \cdots,  k$.
Pick any $\epsilon>0$. 
We have $(\alpha_l^{(i)})_{X'} \to \alpha_{l,X'}$ in $\num^1(X')$ hence
 $(\alpha_l^{(i)})_{X'} \le  \alpha_{l,X'} + \epsilon \om_{X'}\le (\beta_l)_{X'}+ \epsilon \om_{X'}$ for all $i$ large enough, hence
$\alpha_l^{(i)} \le  \beta_l+\epsilon \om_{X'}$ by Lemma~\ref{lem:negativity}. 
This proves
\[\Delta\le  (\alpha^{(i)}_1\cdot \alpha^{(i)}_2 \cdot \ldots \cdot \alpha^{(i)}_k\cdot\gamma)\le (\beta_1\cdot \beta_2 \cdot \ldots \cdot \beta_k\cdot\gamma)  + \epsilon kC^{k-1} (\om^k\cdot \gamma)\]
and we conclude by letting $\epsilon\to0$.

Finally, pick any net of nef Cartier $b$-divisor classes  $\beta^{(i)}_1, \cdots,  \beta^{(i)}_k$
decreasing to  $\alpha_1, \cdots, \alpha_k$, and let $\Delta'$ be the limit of the net of decreasing classes
$(\beta^{(i)}_1\cdot \beta^{(i)}_2 \cdot \ldots \cdot \beta^{(i)}_k \cdot \gamma)$. By~\eqref{eq:inf-nef}
we have $\Delta \le \Delta'$. But the same argument applies reversing the role of the nets  $\alpha^{(i)}_l$ and  $\beta^{(i)}_l$
which implies $\Delta = \Delta'$, and proves~\eqref{eq:inf-nef2}.

We conclude by observing that since any nef $b$-divisor class is the limit of a decreasing net of nef Cartier $b$-divisor classes,
~\eqref{eq:inf-nef} and~\eqref{eq:inf-nef2} characterize a unique class. 
\end{proof}

The following result is the main statement of the section. It can be viewed as a vague "algebraic" analog of Bedford-Taylor's intersection theory. 

\begin{thm}\label{thm:intersection}
\begin{enumerate}
\item
The intersection product of nef divisor classes coincides with the usual intersection when all classes are Cartier. 
It is non-negative, symmetric, multilinear in the variables $\alpha_1, \cdots, \alpha_k\in\nef(\X)$, linear in the variable $\gamma$ and increasing in each variable.
\item
The intersection product of nef classes is upper-semicontinuous: for any net of divisor classes such that $\alpha^{(j)}_i\to \alpha_i$ in the weak topology, then 
we have
\[
\limsup_j \left(\alpha_1^{(j)}\cdot \ldots \cdot \alpha_k^{(j)}\cdot\gamma \right) \le 
(\alpha_1\cdot \ldots \cdot \alpha_k\cdot\gamma)~.
\]
for any class $\gamma \in \BPF^{d-k}(\X)$.
\item
The intersection product of nef divisor classes is continuous along decreasing sequences: if $\alpha^{(j)}_i$ is a net of divisor classes  decreasing  to $\alpha_i$, then 
\[
\lim_j \left(\alpha_1^{(j)}\cdot \ldots \cdot \alpha_k^{(j)} \cdot\gamma\right) = 
(\alpha_1\cdot \ldots \cdot \alpha_k \cdot\gamma)
\]
in $\num_\om^{k+l}(\X)$ where $\gamma \in \BPF^{l}(\X)$.
\item
The intersection product of nef classes is continuous in the $\normom{\cdot}$-norm. In other words, if $\normom{\alpha^{(j)}_i-\alpha_i}\to0$, then we have
\[
\lim_j \left(\alpha_1^{(j)}\cdot \ldots \cdot \alpha_k^{(j)} \cdot\gamma\right) = 
(\alpha_1\cdot \ldots \cdot \alpha_k \cdot\gamma)
\]
in $\num_\om^{k+l}(\X)$.
\end{enumerate}
\end{thm}

\begin{ex}\label{ex:not-continuous} The above intersection product is not continuous with respect to the weak topology. Take $X = \mathbb{P}^2$ and a general line $L$ and choose a sequence of point $p_n\in \mathbb{P}^2$
converging to the generic point (i.e. for any curve $C$, the set $\{n, p_n\in C\}$ is finite).
For each $n$, consider the blow-up $\pi_n\colon X_n \to \mathbb{P}^2$ at $p_n$, and let $\alpha_n$ be 
the nef Cartier class attached to $\pi^*_n (c_1(\cO(1))) - [E_n]$ where $E_n$ is the exceptional divisor of $\pi_n$.
Then the sequence $\alpha_n$  converges weakly to the Cartier class attached to $c_1(\cO(1))$ as $n\rightarrow +\infty$, but $\alpha_n^2=0$ whereas $c_1(\cO(1))^2=1$.
\end{ex}

\begin{proof}
Compatibility with the intersection of Cartier $b$-divisor classes is obvious from the definition.
The symmetry and multilinearity follows from the corresponding properties for the intersection product of divisors. 
The fact that is it is increasing is a consequence of the observation that 
$(\alpha'_1\cdot \alpha_2 \cdot \ldots \cdot \alpha_k \cdot \gamma) \ge (\alpha_1\cdot \alpha_2 \cdot \ldots \cdot \alpha_k \cdot \gamma)$
as soon as $\alpha'_1\ge \alpha_1$. These arguments prove (1).

We fix a class $\om\in\Nef(\X)$ determined
by an ample class in $X$, 
and prove (2). Suppose  $\alpha^{(i)}_l\to \alpha_l$ weakly. 
Pick any class $\Delta$ in the accumulation locus of the net $\left(\alpha_1^{(j)}\cdot \ldots \cdot \alpha_k^{(j)} \cdot\gamma\right)$. 
We may replace the net $I$ by one indexed by the family of open neighborhoods of $\Delta$ (endowed with the partial ordering 
given by the inclusion) and assume that $\left(\alpha_1^{(j)}\cdot \ldots \cdot \alpha_k^{(j)} \cdot\gamma\right) \to \Delta$.

Let $\beta^{(i)}_l$ be nets of nef Cartier $b$-divisor classes
decreasing to $\alpha_l$.
Fix any large index $i$, and choose a model $X'$ in which $\beta^{(i)}_l$, $l=1, \cdots, k$ are all determined. 
Since $(\alpha_l^{(j)})_{X'} \to (\alpha_l)_{X'}$ in $\num^1(X')$ we have
 $(\alpha_l^{(j)})_{X'}\le  (\beta^{(i)}_l)_{X'}  + \epsilon \om$ for an arbitrary small $\epsilon>0$ and all $j$ large enough, hence
$\alpha_l^{(j)} \le  \beta^{(i)}_l + \epsilon \om$ by Lemma~\ref{lem:negativity}. 
We infer
\[
\left(\alpha_1^{(j)}\cdot \ldots \cdot \alpha_k^{(j)} \cdot \gamma \right)
\le
\left(\beta_1^{(l)}\cdot \ldots \cdot \beta_k^{(l)} \cdot \gamma \right)
 + \epsilon kC^{k-1} (\om^k\cdot \gamma)
\]
where $C>0$ is a any constant such that $\beta_l^{(i_0)}\le C\om$, $l=1,\cdots,k$, and $i_0$ is a fixed index.
We now let successively $j$ tend to infinity in the net $I$, then $\epsilon\to0$, and finally let $l$ tend to infinity in $I$.
We conclude that $\Delta \le(\alpha_1\cdot \ldots \cdot \alpha_k \cdot\gamma)$ as required.

Now suppose that $\alpha^{(j)}_i$ is a net of nef $b$-divisor classes (not necessarily Cartier) decreasing to $\alpha_i$. 
We claim that there exists a constant $C>0$ such that for any $\epsilon>0$ 
one has
\[
\left(\alpha_1\cdot \ldots \cdot \alpha_k \cdot \gamma \right)
\le
\left(\alpha_1^{(j)}\cdot \ldots \cdot \alpha_k^{(j)} \cdot \gamma \right)
\le
\left(\alpha_1\cdot \ldots \cdot \alpha_k \cdot \gamma \right)
 + \epsilon kC^{k-1} (\om^k\cdot \gamma)
\]
for all $j$ large enough.
Indeed, the first inequality is a consequence of the 
fact that the intersection product is increasing in each variable, 
and the second inequality follows from the previous proof. 
These estimates imply the convergence $\left(\alpha_1^{(j)}\cdot \ldots \cdot \alpha_k^{(j)} \cdot \gamma \right)
\to
\left(\alpha_1\cdot \ldots \cdot \alpha_k \cdot \gamma \right)$ to hold in $\num_\om^{k+l}(\X)$ which proves (3).

Finally we observe that if $\alpha$ and $\alpha'$ are two nef classes
such that $\normom{\alpha - \alpha'}\le \epsilon$, then
$ - \epsilon \om \le (\alpha - \alpha')\le  \epsilon \om$ 
so that (4) follows from (3).
\end{proof}

\smallskip 

Proposition \ref{prop_Siu}  together with Theorem~\ref{thm:approx2} yield the next two results:
\begin{cor}\label{cor:Siu}
There exists a constant $C_d>0$ depending only on $d$ such that 
for any set of classes $\alpha_1, \cdots, \alpha_k, \beta \in \Nef(\X)$ with $(\beta^d)>0$ and any $\gamma\in\BPF^l(\X)$ we have
\[ 
(\alpha_1\cdot \ldots \cdot \alpha_k\cdot\gamma) 
\le 
C_d\,
\frac
{(\alpha_1\cdot \ldots \cdot \alpha_k\cdot\gamma\cdot\beta^{d-k-l})}{(\beta^d)} \beta^{k+l} ~.
\]
\end{cor}
We obtain a "numerical" or cohomological version of the Chern-Levine-Nirenberg inequalities (\cite{chern_levine_nirenberg} and \cite[(3.3)]{demailly_agbook}). 
\begin{cor} \label{cor:CLN}
There exists a constant $C_d>0$ depending only on $d$  such that for any $\alpha_1, \ldots , \alpha_k \in \Nef(\X)$ and for any $\gamma \in \BPF^{d-k}(\X)$, we have:
\begin{equation*}
0 \leqslant (\alpha_1 \cdot \ldots \cdot \alpha_k \cdot \gamma) \leqslant C_d \normbpf{\alpha_1} \cdot \ldots\cdot \normbpf{\alpha_k} \normbpf{\gamma}.
\end{equation*}
\end{cor}   

\subsection{Diskant inequalities and applications}\label{sec:Diskant}

\begin{thm}
Pick any two classes $\alpha,\beta\in \Nef(\X)$ such that $(\alpha^d)>0$ and $(\beta^d)>0$.
Let $s> 0$ be the largest positive number such that 
$\alpha - s \beta\ge0$.
Then we have:
\begin{equation}\label{eq:Diskant}
(\alpha^{d-1}\cdot\beta)^{\frac{d}{d-1}} - (\alpha^d)(\beta^d)^{\frac{1}{d-1}} \ge 
\left( 
(\alpha^{d-1}\cdot\beta)^{\frac{1}{d-1}}
- s (\beta^d)^{\frac1{d-1}}
\right)^d.
\end{equation} 
\end{thm}

\begin{rem}
The previous theorem was proved by S. D. Cutkosky for nef and big Cartier $b$-divisors  in arbitrary characteristic,~\cite[Theorem~6.9]{MR3368254}.
It was proved for big and movable divisors and X is a projective variety over an algebraically closed field of characteristic zero in~\cite[Proposition 3.3, Remark 3.4]{lehmann_xiao_positivity}.
S.D. Cutkosky extended this result recently to envelopes of big divisors in characteristic zero,~\cite[Theorem~1.4]{cutkosky-preprint}.
\end{rem}

\begin{proof}
Pick any two nets $\alpha_j, \beta_j\in\cNef(\X)$ parametrized by an inductive set $J$ and decreasing
to $\alpha$ and $\beta$ respectively.
Fix any $\epsilon >0$.

Observe that $(\alpha_j^d) \ge (\alpha^d)>0$ and 
 $(\beta_j^d) \ge (\beta^d)>0$ so that $\alpha_j$ and $\beta_j$ are big.
Tthe largest non-negative number $s_j$ such that 
$\alpha_j - s_j \beta_j\ge0$ is thus positive.
By~\cite[Theorem~F]{boucksom_favre_jonsson_volume}, we have
\begin{equation*}
(\alpha_j^{d-1}\cdot\beta_j)^{\frac{d}{d-1}} - (\alpha_j^d)(\beta_j^d)^{\frac{1}{d-1}} \ge 
\left( 
(\alpha_j^{d-1}\cdot\beta_j)^{\frac{1}{d-1}}
- s_j (\beta_j^d)^{\frac1{d-1}}
\right)^d.
\end{equation*} 
In other words, $s_j\ge \tau_j$ where
\[
\tau_j
=
\frac1{(\beta_j^d)^{\frac1{d-1}}}
(\alpha_j^{d-1}\cdot\beta_j)^{\frac{1}{d-1}}
- 
\frac1{(\beta_j^d)^{\frac1{d-1}}}
\left((\alpha_j^{d-1}\cdot\beta_j)^{\frac{d}{d-1}} - (\alpha_j^d)(\beta_j^d)^{\frac{1}{d-1}}\right)^{\frac1d}.
\]
By Theorem~\ref{thm:intersection} (3), we have
\[
\lim_J \tau_j = \tau :=
\frac1{(\beta^d)^{\frac1{d-1}}}
(\alpha^{d-1}\cdot\beta)^{\frac{1}{d-1}}
- 
\frac1{(\beta^d)^{\frac1{d-1}}}
\left((\alpha^{d-1}\cdot\beta)^{\frac{d}{d-1}} - (\alpha^d)(\beta^d)^{\frac{1}{d-1}}\right)^{\frac1d},
\]
and we obtain $\alpha - \tau \beta \ge 0$ as required.
\end{proof}

We now complete the proof of Theorem~\ref{thm:approx-nef}.

\begin{proof}[Proof of Theorem~\ref{thm:approx-nef}]
Let $\alpha$ be any nef $b$-divisor class. By adding $\frac1n \om$ with $\om$ a fixed Cartier $b$-divisor class determined in $X$
by an ample divisor, we may suppose that $\alpha$ is big so that $(\alpha^d)>0$.

Let $\alpha_j$ be a net of nef Cartier $b$-divisor classes decreasing to $\alpha$.
Diskant inequalities yield $\alpha \ge s_j \alpha_j$ with 
\[
s_j:= 
\frac1{(\alpha_j^d)^{\frac1{d-1}}}
(\alpha^{d-1}\cdot\alpha_j)^{\frac{1}{d-1}}
- 
\frac1{(\alpha_j^d)^{\frac1{d-1}}}
\left((\alpha^{d-1}\cdot\alpha_j)^{\frac{d}{d-1}} - (\alpha^d)(\alpha_j^d)^{\frac{1}{d-1}}\right)^{\frac1d}.
\]
Since $(\alpha_j^d)\to (\alpha^d)$, and $(\alpha^{d-1}\cdot\alpha_j)\to (\alpha^d)$, we
may extract a subsequence $i_n$ such that 
$\alpha_{i_n}$ is still decreasing, and $s_{i_n} \to 1$
which implies $\alpha_{i_n} \to \alpha$.
\end{proof}
The same arguments yield:

\begin{cor} \label{cor_best_approx}
For any class $\alpha\in \Nef(\X)$ such that $(\alpha^d)>0$ and for any $\epsilon>0$, 
there exists a class $\alpha'\in\cNef(\X)$ such that 
$(1-\epsilon) \alpha' \le \alpha \le \alpha'$.
\end{cor}

This corollary allows one to obtain  better continuity properties of the intersection product. 

\begin{cor}\label{cor:weak-limits}
For any $\alpha_1, \cdots, \alpha_k\in\Nef(\X)$, the map
\[
\gamma \in \BPF^l(\X) \mapsto (\alpha_1\cdot \alpha_2 \cdot \ldots \cdot \alpha_k\cdot\gamma)\in\BPF^{k+l}(\X)
\]
is weakly continuous. 
\end{cor}

\begin{proof}
When $\alpha_i$ are Cartier, the result is clear. By multilinearity and by replacing $\alpha_i$ by $\alpha_i + \om$, we may suppose that 
$(\alpha_i^d)>0$ for all $i$. Pick any $\epsilon>0$. 
By the previous corollary, one can find  $\alpha'_i\in\cNef(\X)$ such that 
$(1-\epsilon) \alpha_i' \le \alpha_i \le \alpha_i'$. It follows that 
\begin{align*}
0\le
(\alpha'_1 \cdot \ldots \cdot \alpha'_k\cdot\gamma)
-(\alpha_1\cdot \ldots \cdot \alpha_k\cdot\gamma)
&\le 
((1-\epsilon)^{-k} -1) 
(\alpha_1\cdot \ldots \cdot \alpha_k\cdot\gamma)
\\
&\le C \epsilon (\om^k\cdot\gamma)~.
\end{align*}
Suppose $\gamma_n \to \gamma$ weakly in $\BPF^l(\X)$.
Then 
\begin{align*}
\liminf_n
(\alpha_1\cdot \ldots \cdot \alpha_k\cdot\gamma_n)
&\ge 
\liminf_n\, 
(\alpha'_1 \cdot \ldots \cdot \alpha'_k\cdot\gamma_n)
- C \epsilon (\om^k\cdot\gamma_n)
\\
&=
(\alpha'_1 \cdot \ldots \cdot \alpha'_k\cdot\gamma)
- C \epsilon (\om^k\cdot\gamma)
\ge
(\alpha_1 \cdot \ldots \cdot \alpha_k\cdot\gamma)
- C \epsilon (\om^k\cdot\gamma),
\end{align*}
which implies  $\liminf_n
(\alpha_1\cdot \ldots \cdot \alpha_k\cdot\gamma_n)\ge
(\alpha_1\cdot \ldots \cdot \alpha_k\cdot\gamma)$.

We conclude using Theorem~\ref{thm:intersection} (1).
\end{proof}

Recall the definition of $\nums(\X)$ from \S\ref{sec:3Banach}.

\begin{cor} \label{cor:nef_in_nums} 
The following inclusion holds:
\[\Nef(\X)\subset \nums(\X).\] 
More precisely, for any nef class $\alpha\in\Nef(\X)$ there exists a sequence of nef Cartier $b$-divisor classes
$\alpha_n\in\cNef(\X)$ such that 
$\normsigma{\alpha-\alpha_n} \to 0$.

In particular, we have  following continuous injection
\begin{equation*}
(\Vect(\Nef(\X)), \normbpf{\cdot} ) \hookrightarrow (\nums(\X),\normsigma{\cdot}).
\end{equation*}
 \end{cor}

\begin{proof}
Let $\alpha$ be any nef class.
Since any Cartier class belongs to $\nums(\X)$, by adding $\om$ if necessary, we may assume that $(\alpha^d)>0$.
By Corollary \ref{cor_best_approx}, for each $n$, one can take a nef Cartier class $\alpha_n$ such that 
$(1+\frac1n)^{-1} \alpha_n\le \alpha \le \alpha_n$.
Using Corollary~\ref{cor:CLN}, we infer for all integers $m\ge n$:  
\begin{align*}
\normsigma{\alpha_n - \alpha_m}^2
&=
\sup_{ \substack{\gamma \in \cBPF^{d-2}(\X) \\
(\gamma \cdot \om^2) = 1}}   2 ((\alpha_n - \alpha_m)\cdot  \om \cdot \gamma )^2  -  ( (\alpha_n - \alpha_m)^2 \cdot \gamma)
\\
&\le
\sup_{ \substack{\gamma \in \cBPF^{d-2}(\X) \\
(\gamma \cdot \om^2) = 1}}
\frac2{n^2}  (\alpha_m \cdot  \om \cdot \gamma )^2 + \frac2{n} (\alpha_m^2\cdot \gamma)
\le \frac{C}n
\end{align*}
for some constant $C>0$. 
This shows that $(\alpha_n)$ is a Cauchy sequence in $\nums(\X)$. But since it converges weakly to $\alpha$
we conclude that $\alpha$ belongs to $\nums(\X)$ and $\normsigma{\alpha_n-\alpha} \to0$.

By linearity, we get $\Vect(\Nef(\X)) \subset \nums(\X)$ and  the continuity follows from the fact that $\normbpf{\alpha} = (\alpha \cdot \om^{d-1})$ for all $\alpha \in \Nef(\X)$ by Lemma \ref{lem_straightforward}. 
\end{proof}

We now relate the intersection product of nef divisors with the intersection product $\nums(\X) \times \nums(\X) \to \numbpfs{2}(\X)$ recalled in Theorem \ref{thm_intersection_nums}. 
Denote by  $\imath\colon \Nef(X)\to \nums(\X)$ the natural inclusion map given be the previous corollary.

\begin{cor}\label{cor:compat-nums} 
For any two nef $b$-divisor classes $\alpha,\beta$, 
the class $\alpha\cdot \beta\in\BPF^2(\X)$ defined by Proposition-Definition~\ref{prop:exists-key}
is identical to the intersection product  $\imath(\alpha)\cdot\imath(\beta)\in \numbpfs{2}(\X)$
given by Theorem \ref{thm_intersection_nums}. 
\end{cor}

\begin{proof} 
Take two sequences of nef Cartier $b$-divisor classes such that $\normsigma{\alpha_n-\alpha}\to0$ and $\normsigma{\beta_n-\beta}\to0$.
By continuity of the intersection product in $\nums(\X)$, we have $\imath(\alpha)\cdot\imath(\beta) = \lim_n \lim_m \alpha_n\cdot\beta_m$.
By Corollary~\ref{cor:weak-limits}, we get 
$ \lim_m \alpha_n\cdot\beta_m = \alpha_n\cdot \beta$; and then  $ \lim_n \alpha_n\cdot\beta = \alpha\cdot \beta$ which proves the result.
\end{proof}

\subsection{Variations on the Hodge index theorem}

\hodgeindex*
\begin{proof}
In view of Corollary~\ref{cor:compat-nums}, the statement is a direct consequence of the version of Hodge index theorem holding in $\nums(\X)$, see~\cite[Theorem~3.13]{spectral}.
\end{proof}

\begin{prop}\label{prop:KT}
For any two nef classes $\alpha, \beta$, the sequence $e_k:= (\alpha^k\cdot \beta^{d-k})$ is log-concave. 
Moreover if $\alpha$ or $\beta$ are nef and big, and the sequence $(\log e_k)_k$ is affine then 
$\alpha$ and $\beta$ are proportional.
\end{prop}
\begin{proof}
Write $\alpha$ and $\beta$ as decreasing limits of sequences of nef Cartier classes
$\alpha_n\downarrow\alpha$ and $\beta_n\downarrow\beta$. By Khovanskii-Teissier's inequalities, the sequence
$(\alpha_n^k\cdot \beta_n^{d-k})$ is log-concave which implies by Theorem \ref{thm:intersection} the log-concavity of the sequence $e_k=(\alpha^k\cdot \beta^{d-k})$
by letting $n\to\infty$.

Suppose now that $\alpha$ and $\beta$ are big and nef, and that the sequence $e_k$ is linear. 
Scale both classes such that $(\alpha^d)=(\beta^d) =1$ so that $(\alpha^{d-1}\cdot\beta)=1$ too.
Then $\alpha =\beta$ by Diskant's inequalities.
\end{proof}

Recall that given any two classes $\theta, \theta'\in\nums(\X)$
the pairing $\alpha\cdot\beta$ belongs to $\numbpfs{2}(\X)$ which is included in $ \numbpfs{d-2}(\X)^*$ by~\cite[Proposition~2.10]{spectral}.
It follows that for any class $\gamma\in\BPF^{d-2}(\X)$, 
 the pairing $\theta\cdot \theta'\cdot \gamma \in \R$ is well-defined.

\begin{prop}\label{prop:Hodge-Riemann}
Let $\gamma\in\BPF^{d-2}(\X)$, and $\alpha\in\Nef(\X)$. Suppose that $\alpha \cdot \gamma \neq 0$, and that 
there exists a sequence $\gamma_n\in \cBPF^{d-2}(\X)$ decreasing to $\gamma$.

Then the quadratic form 
\[q_\gamma(\theta):= \left(\theta^2 \cdot\gamma\right)\]
 is semi-negative on the
hyperplane $H_\alpha:= \{\theta \in \nums(\X), (\theta \cdot \alpha\cdot\gamma)=0\}$.
In particular $\ker(q_\gamma|_{H_\alpha}) = \{\theta\in H_\alpha,\, q_\gamma(\theta)=0\}$.
\end{prop}

\begin{rem}
We believe the assumption on the existence of a sequence of Cartier BPF classes approximating $\gamma$ to be superfluous. Note that it
is satisfied when $\gamma$ is a positive linear combination of intersections of nef classes or when the codimension of $\gamma$ is $1$ by Theorem \ref{thm:approx-nef}.
\end{rem}

\begin{proof}
Observe that the last statement is a consequence of Cauchy-Schwarz inequality.
Let us prove that $q$ is semi-negative on $H_\alpha$.

Pick any decreasing sequence $\alpha_n\in\cNef(\X)$ converging to $\alpha$.
Pick any sequence of models $X_{n+1}\ge X_n$ such that 
for each $n$, the classes $\gamma_n$ and $\alpha_n$ are determined in $X_n$. By perturbing slightly $\gamma_n$, we may assume 
it is represented by a surface $\Sigma$ in $X_n$.
We may also assume that $\gamma_n \le \omega^{d-2}$ for some fixed ample class $\omega$ in $X$. Note that 
\[
0< c:=
\left(\omega\cdot\alpha  \cdot \gamma \right)
\le
\left(\omega\cdot\alpha_n  \cdot \gamma_n \right)
\le 
C:= \left(\omega\cdot\alpha_1  \cdot \gamma_1 \right)
~.\]
Pick any class $\theta\in \nums(\X)$ such that $(\theta \cdot \alpha\cdot\gamma)=0$.
Take any sequence of Cartier classes $\theta^{(j)}\to \theta$ in $\nums(\X)$. 
By~\cite[Theorem~3.16]{spectral}, we have $(\theta^{(j)} \cdot \beta)\to (\theta\cdot \beta)$ in $\numbpfs{2}(\X)$
for any class $\beta\in\Nef (\X)$. 

Fix $\epsilon>0$, and take $n$ large enough such that  $|(\theta \cdot \alpha_n\cdot \gamma_n)|\le \epsilon$.
For all $j\gg1$, we have $|(\theta^{(j)} \cdot \alpha_n\cdot \gamma_n)|\le 2\epsilon$. 
Set $\tilde{\theta}_j= \theta^{(j)} - \eta_j \omega$ so that $(\tilde{\theta}_j \cdot  \alpha_n\cdot \gamma_n)=0$, and observe that 
\[
0\le |\eta_j| := \frac{|(\theta^{(j)}\cdot  \alpha_n\cdot \gamma_n)|}{\left(\omega\cdot  \alpha_n\cdot \gamma_n\right)}
\le 2c^{-1} \epsilon~.
\]
By the Hodge index theorem, we get 
\[
\left(\tilde{\theta}^2_j\cdot \gamma_n \right) = \left(\tilde{\theta}^2_j|_\Sigma\right)
\le 0~.\]
Since the sequence $\eta_j$ is bounded, we may extract a subsequence converging to some $0<\eta\le 2c^{-1}\epsilon$
and letting $j\to\infty$, we obtain
\[
\left((\theta - \eta \omega)^2 \cdot\gamma \right)
\le 0\]
We conclude by letting first $n\to\infty$ and then $\epsilon\to0$.
\end{proof}

An analog of the Hodge index theorem was used by Dinh-Sibony in \cite{MR2066940} to study commutative subgroups of the group of automorphisms on compact K\"ahler manifolds. 
Following~\cite[Corollary~3.5]{MR2066940}, we extend their result to arbitrary $b$-divisor classes.
\begin{thm}\label{thm:DS}
Pick any class $0\neq \gamma\in\BPF^k(\X)$ such that 
there exists a sequence $\gamma_n\in \cBPF^{k}(\X)$ decreasing to $\gamma$.

If $\theta, \theta' \in\nums(\X)$ satisfy $(\theta^2\cdot\gamma)\ge0$,   $((\theta')^2 \cdot \gamma) \ge0$, and $(\theta \cdot \theta' \cdot \gamma) =0$, 
then there exist $(a, a')\neq (0)$ such that $(a \theta+ a' \theta') \cdot \gamma =0$.
\end{thm}

\begin{rem}
Note that the conditions $(\theta^2\cdot\gamma)\ge0$,   $((\theta')^2 \cdot \gamma) \ge0$ are automatically satisfied when $\theta, \theta'\in \Nef(\X)$.
\end{rem}

\begin{proof}
We may suppose that $\theta\cdot \gamma \neq 0$, $\theta'\cdot \gamma \neq 0$.

We claim that
\[\theta^2\cdot \gamma =
(\theta')^2\cdot \gamma = 0~.\]
To see this pick any class $\beta\in\cBPF^{d-k-2}(\X)$
such that the curve classes $\theta\cdot \gamma \cdot \beta $ and $\theta'\cdot \gamma \cdot \beta $ are both non zero.
Observe that  
\[q_\gamma(a \theta+ a' \theta'):= \left((a \theta+ a' \theta')^2 \cdot \gamma\cdot \beta\right) =a^2 \left(\theta^2\cdot \gamma\cdot\beta\right) + (a')^2 
\left((\theta')^2\cdot \gamma\cdot \beta\right) \ge 0~.\] 
Pick any big and nef class $\alpha\in\cNef(\X)$ such that $(\theta\cdot \gamma \cdot \beta\cdot \alpha) \neq0$, and $(\theta'\cdot \gamma \cdot \beta \cdot \alpha) \neq0$.
By Proposition~\ref{prop:Hodge-Riemann}, we have $q_\gamma\le 0$ on the hyperplane $H:= (\gamma\cdot\beta\cdot\alpha)^{\perp}$. 
One can thus find a non-zero element  $a\theta + a'\theta' \in H$ with $a, a'\neq 0$ such that $q_\gamma(a\theta + a'\theta')=0$, and it follows
that
$\left(\theta^2\cdot \gamma\cdot\beta\right) =
\left((\theta')^2\cdot \gamma\cdot \beta\right) = 0$.
Since $\beta$ is arbitrary, the claim is proved.

\smallskip

Pick any class $\beta\in\cBPF^{d-k-2}(\X)$ such that both curve classes $\theta\cdot \Omega$ and $\theta'\cdot \Omega$ are non zero with $\Omega=\beta\cdot\gamma$.
We shall prove that $\theta\cdot \gamma\cdot\beta = c(\beta)\, \theta'\cdot \gamma\cdot\beta$ for some $c(\beta)>0$.
By adding a small multiple of $\om^k$ to $\gamma_n$ and $\om^{d-k-2}$ to $\beta$, we may assume that  $\Omega_n:= \gamma_n\cdot\beta$
is strongly BPF, and  determined in some model by a surface $\Sigma_n$ which is the image under a flat map of a complete intersection ample divisor.
We may also inforce that $\theta\cdot\Omega\cdot\om \neq0$.

Choose any $t_n$ such that $\left(\theta_n^2\cdot \Omega_n\right) =0$ with $\theta_n = \theta + t_n \om$.
Observe that
\begin{align*}
0
=
\left(\theta_n^2\cdot \Omega_n\right) 
&= 
\left(\theta^2\cdot \Omega_n\right) 
+ 2 t_n \left(\om\cdot \theta\cdot \Omega_n\right) 
+t_n^2 \left(\om^2\cdot \Omega_n\right)~.
\\
&=\left(\theta^2\cdot (\Omega_n-\Omega)\right) 
+ 2 t_n \left(\om\cdot \theta\cdot \Omega_n\right) 
+t_n^2 \left(\om^2\cdot \Omega_n\right)~.
\end{align*}
Since $\left(\om^2\cdot \Omega_n\right)$ is uniformly bounded from below by $(\om^2 \cdot \gamma \cdot \beta )$, and $\left(\theta^2\cdot \Omega_n\right)\sim \left(\theta^2\cdot \Omega\right) =0$, 
such a constant $t_n$ exists and we may choose $t_n\to0$ as $n\to\infty$. 

Now pick another big and nef class $\om'$ such that 
$\om \cdot \Omega$ and $\om' \cdot \Omega$ are not proportional, normalized by 
$(\om \cdot \Omega \cdot\theta)= (\om' \cdot \Omega \cdot\theta)\neq0$.
Choose $r_n, r'_n$ such that
\begin{align}
0
&=\left(\theta_n'\cdot \theta_n\cdot \Omega_n \right) 
 \label{eq:first}
\\
0&=((\theta_n')^2\cdot \Omega_n)
 \label{eq:second}
\end{align}
with $\theta'_n = \theta' + r_n \om+ r'_n \om'$.

\begin{lem}\label{lem:rn}
One can find constants $r_n , r_n'$ satisfying~\eqref{eq:first} and~\eqref{eq:second} 
converging to zero as $n\to\infty$.
\end{lem}

From $\theta_n^2\cdot\Sigma_n = (\theta'_n)^2\cdot\Sigma_n =\theta_n\cdot\theta'_n\cdot\Sigma_n=0$
we infer that $\theta_n\cdot\Sigma_n$ and  $\theta'_n\cdot\Sigma_n$ are proportional by the usual Hodge index theorem.
Letting $n\to\infty$ proves that  $\theta\cdot \gamma\cdot\beta = c(\beta)\, \theta'\cdot \gamma\cdot\beta$ for some $c(\beta)>0$.

\smallskip

We claim that the proportionality constant $c(\beta)$ is in fact independent on $\beta$ so that 
$\theta\cdot \gamma = c\, \theta'\cdot \gamma$ for some $c>0$ as required. 

To prove our claim we suppose first that $\beta = \om_1\cdot \om_2\cdot\ldots \cdot \om_{d-k-2}$, and 
$\beta' = \om'_1\cdot \om_2\cdot \ldots \cdot \om_{d-k-2}$ for some big and nef classes
$\om'_1, \om_1, \ldots, \om_{d-k-2}$. Then 
\[
c(\beta)\, (\theta'\cdot \gamma\cdot\beta\cdot\om'_1)
=
(\theta\cdot \gamma\cdot\beta\cdot\om'_1)
=
(\theta\cdot \gamma\cdot\beta'\cdot\om_1)
=
c(\beta')\, (\theta'\cdot \gamma\cdot\beta\cdot\om'_1)
\]
hence $c(\beta) =c(\beta')$ in this case. Since flat push-forwards of complete intersection classes of ample divisors generate $\cNef(\X)$,
we get the result (compare with the proof of~\cite[Theorem~3.13]{spectral}).
\end{proof}

\begin{proof}[Proof of Lemma~\ref{lem:rn}]
Write $r=r_n$ and $r'=r'_n$. Note that~\eqref{eq:first} implies $r'= \epsilon_n - \kappa_n r$
with
\[
\epsilon_n = -\frac{(\theta'\cdot\theta_n\cdot\Omega_n)}{(\om'\cdot\theta_n\cdot\Omega_n)} \to 0
\]
and
\[
\kappa_n = \frac{(\om\cdot\theta_n\cdot\Omega_n)}{(\om'\cdot\theta_n\cdot\Omega_n)} \to 1.
\]
Replacing $r'$ into~\eqref{eq:second}, we obtain the following second degree equation:
\begin{multline*} 
r^2 \left[(\om^2\cdot\hat{\Omega}_n) + \kappa_n^2 ((\om')^2\cdot\hat{\Omega}_n) - 2\kappa_n (\om\cdot\om'\cdot\hat{\Omega}_n)
\right]
\\+ 
2r
\left[
(\om\cdot\theta'\cdot\hat{\Omega}_n) 
+
\epsilon_n(\om\cdot\om'\cdot\hat{\Omega}_n) 
- \kappa_n (\om'\cdot\theta'\cdot\hat{\Omega}_n)
\right]
\\+ 
\left[
((\theta')^2\cdot\hat{\Omega}_n) + \epsilon_n^2 ((\om')^2\cdot\hat{\Omega}_n) + 2\epsilon_n (\om'\cdot\theta'\cdot\hat{\Omega}_n)
\right]=0
\end{multline*}
with $\hat{\Omega}_n=\theta_n\cdot\Omega_n$.
Since the constant term tends to $0$ and the leading term is bounded away from $0$ by Hodge index theorem (using that $\om \cdot \Omega$ and $\om'\cdot \Omega$ are not proportional), we may choose a solution
$r_n\to0$, and this forces $r'_n\to0$.
\end{proof}

\section{Minkowski's problem and applications}\label{sec:minkowski}

Let $\om\in\cNef(\X)$ be any Cartier $b$-divisor class determined by an ample class in $X$. 

\subsection{Solving $\alpha \mapsto \alpha^{d-1}$}

\MAmp*

\begin{rem} 
The map $\alpha \mapsto  \alpha^{d-1}$ is continuous for the $\normbpfs{\cdot}$-norm (hence for the
$\normsigma{\cdot}$-norm too) and along decreasing limits. 
It is not continuous in the weak topology (pull-back Example~\ref{ex:not-continuous} by a map $C \times \Pp^2 \to \Pp^2$ where $C$ is any projective curve). 
\end{rem}

\begin{proof}
Let us prove that $M(\alpha) := \alpha^{d-1}$ is injective on the cone of big and nef classes. Recall that by Corollary~\ref{cor:Siu}, a nef class $\alpha$ satisfies  $\alpha\ge c\om$ for some $c>0$
iff $(\alpha^d)>0$.

If $\alpha$ and $\beta$ are two nef classes such that $\alpha^{d-1}= \beta^{d-1}$, then 
$(\alpha^{d}) = (\alpha\cdot \beta^{d-1})$, and $(\alpha^{d-1}\cdot \beta) = (\beta^d)$. By the log-concavity of the sequence
$k\mapsto \log (\alpha^k\cdot \beta^{d-k})$ we obtain $0<(\alpha^d)= (\beta^d) = (\alpha^k\cdot \beta^{d-k})$ for all $k$.
This  implies the equality between $\alpha$ and $\beta$ by Proposition~\ref{prop:KT}. 

\medskip

We now prove the surjectivity. Fix any class $\gamma\in\BPF^{d-1}(\X)$ such that $\gamma \ge c \om^{d-1}$ for some $c>0$.
We follow the classical variational method to solve a Monge-Amp\`ere equation. Let $K:= \{ \alpha \in \Nef(\X) \ | \  (\alpha \cdot \gamma) = 1\}$.
Note that $K$ is  compact for the weak topology of $\wnum^1(\X)$ by~\cite[Proposition~2.13]{spectral}, as $K\subset \{ \alpha \in \Nef(\X), (\alpha \cdot \om^{d-1}) \le c^{-1}\}$.
Since the function $\alpha \mapsto (\alpha^{d})$ is upper-semicontinuous by Theorem~\ref{thm:intersection} (2), there exists
a class $\alpha_\star\in K$ such that $ (\alpha_\star^{d})= \sup_K (\alpha^d)$.

Observe that $ (\alpha_\star^{d})>0$. Pick any Cartier class $\beta\in\cnum^1(\X)$. For $t$ small enough, the class $\alpha_\star + t\beta$ is pseudo-effective by Siu's inequality. 
Consider the nef envelope $P(\alpha_\star + t\beta)$ as defined on p.\pageref{p:nef-envelope}.

Pick any big and nef class $\om'\in\cNef(\X)$ such that $\alpha_\star \le \omega'$, and $\omega' \pm \beta$ are both nef (any sufficiently large ample class in a model  in which $\beta$ is determined does the job). 
Using ~\cite[Corollary~3.4]{boucksom_favre_jonsson_volume} and by taking limits of Cartier nef classes decreasing to $\alpha_\star$, we have
\[
P(\alpha_\star+ t\beta)^{d}\ge (\alpha_\star^{d}) + d t \left(\alpha_\star^{(d-1)}\cdot \beta\right) - C t^2\]
for some $C>0$ which depends only on $(\omega^d)$ and for any $0 \leq t \leq 1$. 
Observe now that 
\[\alpha_t := \frac{P(\alpha_\star+ t\beta)}{(P(\alpha_\star+ t\beta)\cdot \gamma)} \in K\, \]
 which implies
\begin{align*}
(\alpha_\star^d)
\ge
\frac{P(\alpha_\star+ t\beta)^{d}}{(P(\alpha_\star+ t\beta)\cdot \gamma)^d}
&\ge
\frac{(\alpha_\star^{d}) + d t \left(\alpha_\star^{(d-1)}\cdot \beta\right) - C t^2}{((\alpha_\star+ t\beta)\cdot \gamma)^d}
\\
&
\ge
(\alpha_\star^{d}) + d t \left(\alpha_\star^{(d-1)}\cdot \beta\right) - d t (\beta\cdot \gamma) (\alpha_\star^{d}) +O(t^2)
\end{align*}
hence
$\left(\alpha_\star^{(d-1)}\cdot \beta\right) \le (\beta\cdot \gamma) (\alpha_\star^d)$. Since this is true for $\pm \beta$
we infer $\left(\alpha_\star^{(d-1)}\cdot \beta\right) = (\beta\cdot \gamma) (\alpha_\star^d)$. 
Rescaling $\alpha_\star$, we find a nef class such that $(\alpha^{d-1}) = \gamma$ as required.
\end{proof}

\subsection{An example of $(d-1)$-th root of a Cartier class which is not Cartier}\label{sec:not-Cartier}

For sake of completeness, we include the following clever observation of
Chi Li~\cite[Lemma~4.2]{chi-li}.  We are indebted to S. Boucksom and M. Jonsson 
for pointing this reference to us, and to let us include it here. 

\begin{prop}[Chi Li]
For any big line bundle  $L\to X$, the $b$-numerical curve class $P([c_1(L)])^{d-1}\in \BPF^{d-1}(\X)$ is Cartier.
\end{prop}
\begin{proof}
Let $\pi\colon X'\to X$ be any composition of blow-ups of smooth centers, and let $E$ be any $\pi$-exceptional divisor in $X'$.
It follows from~\cite[Theorem~B]{boucksom_favre_jonsson_volume} that
\[ 
P([\pi^*c_1(L)])^{d-1}\cdot E 
=
\lim_{n\to\infty} \frac{(d-1)!}{n^{d-1}} \dim_\C (V_n)
\]
where
\[V_n:=
\mathrm{Im} 
\left( 
H^0(X',n\pi^*L) \to H^0(E,n\pi^*L|_E)
\right)~.
\]
Observe that 
\[
V_n\subset 
H^0(E, n\pi^*L|_E)
=
H^0(\pi_*E, nL|_{\pi_*(E)})
\]
and since $\dim(\pi_*(E)) < d-1$, asymptotic Riemann-Roch's theorem implies $\dim_\C (V_n)= o(n^{d-1})$ hence
$P([\pi^*c_1(L)])^{d-1}\cdot E =0$.

Define the class $\alpha := P([c_1(L)])^{d-1}$. 
The previous computation shows 
\[(\alpha_{X'}\cdot E)=(\alpha\cdot [E])
= ( P([c_1(L)])^{d-1}\cdot [E])
= ( P([\pi^*c_1(L)])^{d-1}\cdot [E])= 
0\] for any $\pi$-exceptional divisor $E$.
Since $\pi$ is a composition of blow-ups with smooth centers, we have
$\alpha_{X'} = \pi^* \pi_*(\alpha_{X'})= \pi^*\alpha_{X}$ hence $\alpha$
is a Cartier class determined in $X$.
\end{proof}

\begin{ex}
It follows from the work of N. Nakayama~\cite[Theorem~IV.2.10]{nakayama} that there exists a big line bundle on a projective $4$-fold $L\to X$ such that 
$P(L)$ is a nef and big $b$-divisor class which is \emph{not} Cartier.  The $b$-numerical curve class $\alpha= P(L)^3$ is BPF and Cartier whereas its unique third-root is
not.
\end{ex}

\begin{rem} 
It is not difficult to find a nef class $\alpha$ that is not Cartier and such that $\alpha^{d-1}=0$, for example in a toric situation. 
 \end{rem}

\subsection{Proof of Theorem \ref{thmint_injections}}

Recall that we aim at proving the inclusions
\[
\numbpf{1}(\X) \subset \Vect(\Nef(\X)) \subset \nums(\X) \subset \numbpfs{1}(\X) \subset \wnum^1(\X). 
\]
The first and last inclusions hold by definition, see~\cite[\S 2]{spectral}. The second inclusion is Corollary~\ref{cor:nef_in_nums}.
It thus remains to prove the continuous inclusion $ \nums(\X) \subset \numbpfs{1}(\X)$. 
This fact follows from the next theorem of independent interest.
\begin{thm}\label{thm:nums-bpfs}
For any  $\alpha\in\cnum^1(\X)$, and for any $\gamma\in\BPF^{d-1}(\X)$, we have
\begin{equation}\label{eq:bpf-nums}
|(\alpha\cdot \gamma)|
\le 5 (\gamma\cdot\om^{d-1}) \, \normsigma{\alpha}
\end{equation}
\end{thm}
Let  $\alpha_n\in\cnum^1(\X)$ be any Cauchy sequence 
for the norm  $\normsigma{\cdot}$. Then~\eqref{eq:bpf-nums} implies that $\alpha_n$ is a Cauchy sequence for $\normbpfs{\cdot}$,
proving  the continuous injection  $ \nums(\X) \subset \numbpfs{1}(\X)$.

\begin{proof}
Pick any $\gamma \in\BPF^{d-1}(\X)$. By Theorem~\ref{thm:MA}, for any $t>0$ we may find a class
$\delta_t\in\Nef(\X)$ such that $\gamma + t\om^{d-1} = (\delta_t^{d-1})$. Let $s = (\delta_t^{d-2}\cdot\om^2)$.
We now estimate using~\cite[Theorem~3.11]{spectral}:
\begin{align*}
\left|
(\alpha\cdot (\gamma+t\om^{d-1}))
\right|^2
&=
\left|
(\alpha\cdot \delta_t^{d-1})
\right|^2
\\
&\le 
9
q_{\om, \frac1s \delta_t^{d-2}}(\alpha)
q_{\om, \frac1s \delta_t^{d-2}}(s \delta_t)
\\
&\le 
18
( (\gamma+t\om^{d-1}) \cdot\om)^2 \normsigma{\alpha}^2,
\end{align*}
and we conclude by letting $t\to0$.
\end{proof}

\begin{cor}
For any $\alpha,\beta \in\nums(\X)$, the maps 
\[\gamma \in\BPF^k(\X)\mapsto \alpha\cdot \beta\cdot\gamma\in\BPF^{k+2}(\X)
\text{ and } \gamma \in\BPF^k(\X)\mapsto \alpha\cdot \gamma\in\BPF^{k+1}(\X)
\]
are weakly continuous.
\end{cor}
\begin{proof}
Pick $\alpha\in\nums(\X)$ and choose a sequence $\alpha_n\in\cnum^1(\X)$ such that 
$\normsigma{\alpha-\alpha_n} \to0$.
Suppose $\gamma_i\in\BPF^k(\X)$ is a net of BPF classes converging weakly to $\gamma$. We may suppose that 
$C:= \sup_i(\gamma_i\cdot \om^{d-k})< \infty$.
Pick any BPF Cartier $b$-numerical class $\beta\in\cBPF^{d-k-1}(\X)$. 
By~\eqref{eq:bpf-nums} and Corollary~\ref{cor:Siu}, we have
\[
|(\alpha\cdot\gamma_i\cdot\beta) -(\alpha_n\cdot\gamma_i\cdot\beta) | \le 5 (\gamma_i\cdot\beta\cdot \om^{d-k})\normsigma{\alpha-\alpha_n} 
\le CC_d (\beta\cdot \om^{k+1})\,\normsigma{\alpha-\alpha_n} 
~.\]
Pick any $\epsilon>0$. Choose $n_0$ large enough such that $CC_d (\beta\cdot \om^{k+1})\,\normsigma{\alpha-\alpha_{n_0}} \le\epsilon$. 
Since the class $\alpha_{n_0}$ is Cartier, the map $\eta \mapsto \alpha_{n_0}\cdot \eta$ is weakly continuous, hence for $i$ large enough
we have
\[
|(\alpha\cdot\gamma_i\cdot\beta) -(\alpha\cdot\gamma\cdot\beta) | \le 
2\epsilon+ |(\alpha_{n_0}\cdot\gamma_i\cdot\beta) -(\alpha_{n_0}\cdot\gamma\cdot\beta) | \le 3\epsilon~.
\]
This proves $\gamma \in\BPF^k(\X)\mapsto \alpha\cdot \gamma$ is weakly continuous.

The fact that $\gamma\mapsto \alpha\cdot \beta\cdot\gamma$ is continuous follows by the same token and~\cite[Theorem~3.16]{spectral}.
\end{proof}

\subsection{Approximation of  classes in $\BPF^{d-1}(\X)$}
The resolution of the operator $\alpha\mapsto\alpha^{d-1}$ yields interesting consequences for the structure of $\BPF^{d-1}(\X)$.

\begin{cor}\label{cor:approx-curve}
For any $\alpha\in\BPF^{d-1}(\X)$ there exists a sequence of classes $\alpha_n\in\cBPF^{d-1}(\X)$
which is decreasing to $\alpha$. If moreover $\alpha$ is big, then for any $\epsilon>0$ there exists
a class $\alpha'\in\cBPF^{d-1}(\X)$ such that 
$(1-\epsilon) \alpha' \le \alpha \le \alpha'$.
\end{cor}

\begin{proof}
We may suppose that $\alpha\le \om^{d-1}$. 
We claim that for each integer $n$, there exists a class $\alpha_n\in\cBPF^{d-1}(\X)$
such that 
\begin{equation}\label{eq:better}
\alpha+ \frac1n \om^{d-1}
\le \alpha_n \le 
(1+\frac1{2^n}) \left(\alpha+ \frac1n \om^{d-1}\right)
\end{equation}
Since for all $n\ge2$
\begin{align*}
\alpha_{n+1}
&\le 
(1+\frac1{2^{n+1}}) \left(\alpha + \frac1{n+1} \om^{d-1}\right)
\\
&\le 
\alpha_n - \frac1{n(n+1)} \om^{d-1}
+\frac1{2^{n+1}} \left(\alpha + \frac1{n+1} \om^{d-1}\right)\le \alpha_n
\end{align*}
the sequence $\alpha_n$ is decreasing to $\alpha$ as required. When $\alpha$ is big, then for any $\epsilon>0$, we have
$\om^{d-1}/n \le \epsilon \alpha$ for all $n$ large enough, hence it remains to prove~\eqref{eq:better}.

By Theorem~\ref{thm:MA}, one can find a nef and big class $\beta' \in\Nef(\X)$ such that 
$(\beta')^{d-1} = \alpha + \frac1n \om^{d-1}$. 
By Corollary~\ref{cor_best_approx}, one can find a nef Cartier class $\beta(\epsilon)$ such that $\beta' \le \beta(\epsilon) \le (1+\epsilon) \beta'$ 
and we conclude by taking $\alpha_n := \beta(\epsilon)^{d-1}$ with 
$(1+\epsilon)^{d-1} \le 1+ \frac1n$.
\end{proof}
Recall from the definition of the space $\nums(\X)$ ~\cite[\S 3.2]{spectral} that the topological dual of the Banach space $\nums(\X)$ is characterized as follows:
\[
\nums(\X)^* = \left\{ \alpha \in \wnum^{d-1}(\X), \, \exists C>0, | (\alpha\cdot \beta)| \le C \normsigma{\beta} \text{ for all } \beta\in\cnum^1(\X)\right\}.
\]
Since~\eqref{eq:bpf-nums} holds for any class $\alpha\in\nums(\X)$ by density, we get:
\begin{cor}\label{thm:pairing-dim1}
We have
\begin{equation}\label{eq:inclu}
\BPF^{d-1}(\X) \subset \nums(\X)^*.
\end{equation}
\end{cor}


\section{Movable cones and  Lehmann-Xiao's decompositions of curve classes}

In \S \ref{sec:movable}, we discuss the notion of movable classes as defined 
by Fulger and Lehmann and relate it to
positivity properties of $b$-numerical classes. In the last two sections \S\ref{sec:LX} and \S\ref{sec:d-root}, we compare the $(d-1)$-th root
of a BPF $b$-curve class given by Theorem~\ref{thm:MA} to two divisor classes introduced by Lehmann and Xiao
in~\cite{MATH06656980} and~\cite{lehmann_xiao_positivity}.

\subsection{Movable classes}\label{sec:movable}

Let us first introduce the notion of movable classes in the sense of Fulger and Lehmann, see~\cite[\S3]{fulger_lehmann_zariski}.
An effective $k$-cycle $Z$ in $X$ is said to be strictly movable if there exists a variety $\Lambda$ and a reduced closed subscheme $V$ of $X\times \Lambda$
such that for each irreducible component $V_i$ of $V$:
\begin{itemize}
\item
the first projection map $V_i \to X$ is dominant for each $i$;
\item
the second projection map $\pi_i\colon V_i \to \Lambda$ is flat and dominant of relative dimension $k$ for each $i$;
\item
and
$Z= \sum a_i [\pi_i^{-1}(\lambda)]$ for some $\lambda \in \Lambda$ and some $a_i \in \N$.
\end{itemize}
A class $\alpha\in\num^k(X)$ is strictly movable if it is fundamental class of a strictly movable $k$-cycle.
Roughly speaking, a strictly movable class is one represented by a $k$-cycle which 
moves in a family covering $X$.
\begin{defi}\label{defi:movable}
The movable cone is the closed convex cone in $\num^k(X)$ generated by strictly movable classes. 
\end{defi}

It is not difficult to see that this notion coincides with the one given in \S\ref{sec:numerical cycles} when $k=1$. 
It also follows from~\cite[Proposition~3.13]{fulger_lehmann_zariski}  
that the movable cone is equal to the closure of $\sum_\pi \pi_* \BPF^k(X')$ where $\pi\colon X' \to X$
varies over all (smooth) birational models over $X$.

We have the following interpretation of movable classes in terms of $b$-numerical classes.
\begin{thm}\label{thm:interpret-BPF}
For any class $\alpha\in\num^k(X)$ the following are equivalent: 
\begin{enumerate}
\item
$\alpha$ is movable;
\item
there exists a $b$-numerical class $\beta\in\BPF^k(\X)$ such that $\beta_X = \alpha$.
\end{enumerate}
In particular, a $b$-numerical class $\beta\in \wnum^{k}(\X)$ is BPF iff
$\beta_{X'}$ is movable for any model $X'$.
\end{thm}
\begin{rem}
Since the cone $\BPF^{d-1}(\X)$ is the (weak) closure of classes
$[\alpha_{X'}]^{d-1}$ with $\alpha_{X'}\in \num^1(X')$ ample, 
we see that  Definition~\ref{defi:movable} is also consistent with the one given in \S\ref{sec:numerical cycles} 
in the case $k=d-1$.
\end{rem}

\begin{proof}

Suppose  $\alpha\in\num^k(X)$  is movable.  By~\cite[Proposition~3.13]{fulger_lehmann_zariski}, we can find a sequence of models $\pi_n\colon X_n \to X$
and classes $\beta_n\in \BPF^k(X_n)$ such that $(\pi_n)_*\beta_n \to \alpha$. 
Note that $(\beta_n\cdot \om^{d-k})$ is uniformly bounded, hence Banach-Alaoglu's theorem implies
that $F_p= \overline{\{\beta_n, n\ge p\}}$ is weakly compact for all $p$. 
Pick any class $\beta\in\cap_{p\ge1} F_p$. 
Then $\beta$ lies in $\BPF^k(\X)$ and 
we have $\beta_X = \alpha$ which proves (1)$\Rightarrow$(2).

Conversely, suppose  $\beta\in\BPF^k(\X)$. By definition, we can find a sequence of classes $\beta_n\in \cBPF^k(X)$
such that $(\beta_n)_X \to \beta_X$. Since $(\beta_n)_X$ is the image under a proper birational morphism of a BPF class, 
it is movable, hence $\beta_X$ is movable too. This proves (2)$\Rightarrow$(1).

Finally pick any class $\beta\in \wnum^{k}(\X)$, and suppose that $\beta_{X'}$ is movable for any model $X'$.
By what precedes, for each $X'$ there exists a class $\tilde{\beta}(X')\in\BPF^k(\X)$ such that 
$\tilde{\beta}(X')_{X'}=\beta_{X'}\in\num^k(X')$. 
This implies that the net $\tilde{\beta}(X')$ converges (as $b$-numerical classes) to $\beta$, hence
$\beta$ is BPF.
\end{proof}

An easy consequence of the previous theorem is the following analog of Theorem~\ref{thm:cartier_nef}.

\begin{prop} \label{pro:cartier_nef_curve} 
We have \[\cnum^{d-1}(\X) \cap \BPF^{d-1}(\X) = \cBPF^{d-1}(\X).\] 
In other words any Cartier  $b$-numerical curve class that is BPF as a $b$-numerical curve class is determined by a BPF class in some model.
 \end{prop}

\begin{proof}
Pick any $\alpha \in \cBPF^{d-1}(\X)$, and suppose that 
$\alpha$ is determined in some model $X'$ and $\alpha_{X'}\in \num^{d-1}(X')$ is not BPF. By the previous result, $\alpha_{X'}$ is not movable, hence
by the duality theorem \cite[Corollary 2.5]{boucksom_demailly_paun_peternell}, one can find a pseudoeffective class $\beta \in \num^1(X')$
such that $(\alpha_{X'}\cdot\beta) <0$. 

Observe that the Cartier $b$-divisor class $[\beta]$ is psef since the pull-back of any pseudoeffective
divisor remains psef. We then obtain $(\alpha\cdot [\beta]) < 0$ 
which implies that $\alpha$ cannot belong to $\BPF^{d-1}(\X)$.
\end{proof}

Recall that a $b$-numerical class $\alpha$ is psef iff $\alpha_{X'}\ge0$ for all model $X'$. When it is the case, we write $\alpha\ge0$.
Beware that in general a Cartier class which is determined by a psef class is not psef (as a $b$-numerical class).

We conclude this section by exploring duality statements between the BPF and the psef cone in the space of $b$-numerical classes.

\begin{prop}\label{prop:dual-easy}
Pick any $b$-numerical class $\alpha\in\wnum^k(\X)$. 
If $\alpha$ is psef (resp. BPF), then for any BPF (resp. psef) Cartier $b$-numerical class $\beta$
we have 
$(\alpha\cdot\beta)\ge0$.
\end{prop}
\begin{proof}
Suppose $\alpha$ is psef and pick any BPF Cartier $b$-numerical class $\beta$ determined in a 
model $X'$. Then $(\alpha\cdot\beta)= (\alpha_{X'} \cdot\beta_{X'})\ge 0$.

Pick now any psef Cartier $b$-numerical class $\beta$  determined in a 
model $X'$. If $\alpha$ is a Cartier BPF class determined in a model $X''$
then $(\alpha_{X'} \cdot\beta_{X'})=(\alpha\cdot\beta)= (\alpha_{X''} \cdot\beta_{X''})\ge 0$.
We obtain    $(\alpha\cdot\beta)=(\alpha_{X'} \cdot\beta_{X'})\ge0$ by density for \emph{any} nef $b$-numerical classes.
\end{proof}

Recall from the introduction that we stated a duality  between the cones of BPF $b$-classes and psef $b$-classes (Conjecture~\ref{conj:dual-BPF}). 
We prove it here for curve and divisor classes.

\begin{prop}\label{prop:dual-BPF}
For $k=1$ and $k=d-1$, the following two statements hold.
\begin{itemize}
\item
A $b$-numerical class $\alpha\in\wnum^k(\X)$ is psef
iff for any  $b$-numerical class $\beta\in\cBPF^{d-k}(\X)$ 
 we have $(\alpha\cdot\beta)\ge0$.
\item
A $b$-numerical class $\alpha\in\wnum^k(\X)$ belongs to $\BPF^k(\X)$ iff for any psef Cartier $b$-numerical class $\beta\in\cnum^{d-k}(\X)$,  we have $(\alpha\cdot\beta)\ge0$.
\end{itemize}
\end{prop}

\begin{rem}
The fact that a $b$-divisor class is nef if it intersects non-negatively psef $b$-numerical curve classes
is due to B.~Lehmann who formulated this fact in terms of movable classes, see~\cite[Theorem~1.5]{movable-class}.
\end{rem}

\begin{proof}
In view of Proposition~\ref{prop:dual-easy}, we need to prove four implications.

Suppose first that a class $\alpha \in\wnum^1(\X)$ intersects non-negatively
all $b$-numerical classes $\beta\in\cBPF^{d-1}(\X)$. By Theorem~\ref{thm:interpret-BPF}, for any model $X'$ the class
$\alpha_{X'}$ lies in the dual cone of movable curve classes, hence it is psef. This proves $\alpha\ge0$.

\smallskip

Suppose next that $\alpha \in\wnum^{d-1}(\X)$ intersects non-negatively all nef Cartier $b$-divisor classes. 
Since the nef cone $\Nef(X')$ is dual to the psef cone in $\num^{d-1}(X')$, the class $\alpha_{X'}$ is effective 
for all $X'$ hence $\alpha \ge0$.

\smallskip

Take now a class $\alpha \in\wnum^{d-1}(\X)$ intersecting non-negatively all psef $b$-divisor classes. 
Pick any model $X'$ and any effective divisor class $\gamma\in \num^1(X')$.
Since the pull-back of an effective divisor remains effective, the Cartier class $[\gamma]$ is psef, 
hence $(\alpha_{X'}\cdot\gamma)= (\alpha\cdot[\gamma]) \ge0$. 
Since the psef cone is dual to the movable cone in $\num^{d-1}(X')$, we conclude that 
$\alpha_{X'}$ is movable for all $X'$, hence $\alpha$ is BPF by Theorem~\ref{thm:interpret-BPF}.

\smallskip

It remains to prove that a $b$-divisor class is nef iff it intersects non-negatively 
all psef Cartier $b$-numerical curve classes. 
Denote by $\mathcal{C}\subset\wnum^1(\X)$ the dual of the cone of psef Cartier $b$-numerical curve classes.
It is weakly closed and contains $\Nef(\X)$.
Suppose there exists a class $\alpha\in \mathcal{C}$ which is not nef. By Hahn-Banach's theorem
there exists a linear form $\ell\in\wnum^1(\X)^*$ which is continuous for the weak topology and such that 
$\ell(\alpha) <0$ and $\ell|_{\Nef(\X)} \ge 0$. 

\begin{lem}\label{lem:continuous-dual-wN1}
For any weakly continuous linear form $\ell\in\wnum^1(\X)^*$, there exists a Cartier $b$-numerical curve class $\beta$ such that 
$\ell(\alpha)=(\alpha\cdot\beta)$ for all $\alpha\in\wnum^1(\X)$.
\end{lem}

Denote by $\beta$ the Cartier $b$-numerical curve class given by the lemma. Then we have
$(\alpha\cdot\beta) <0$ and $(\theta\cdot \beta)\ge 0$ for all $\theta \in \Nef(\X)$.
By duality, we infer $\beta_{X'}\ge0$ for all model $X'$ which proves $\beta\ge0$. This gives a contradiction and concludes the proof.
\end{proof}

\begin{proof}[Proof of Lemma~\ref{lem:continuous-dual-wN1}]
By Poincaré duality, for any smooth model $X'$,  there exists a unique curve class $\beta_{X'}\in\num^{d-1}(X')$ such that 
$\ell([\alpha])=(\alpha\cdot[\beta_{X'}])$ for all $\alpha\in\num^1(X')$. If $\pi\colon X''\to X'$ is a proper birational morphism, 
and $\alpha$ is a class in $\num^1(X')$ then $[\pi^*\alpha] =[\alpha]$ which implies $\pi_*(\beta_{X''})= \beta_{X'}$.
We now argue that the class $\beta = (\beta_{X'})_{X'\in\cM}$ is Cartier. 

Suppose by contradiction that it is not. Then for any model $X'$, there exists a proper birational morphism  $\pi\colon X''\to X'$ such that 
$\beta_{X''} \neq \pi^* \beta_{X'}$. Using Chow's lemma, we may build a sequence of birational morphisms $\pi_n \colon X_{n+1}\to X_n$
with $X = X_0$
such that $\pi_n$ is a composition of blow-ups along smooth centers, and $\beta_{X_{n+1}} \neq \pi_n^* \beta_{X_n}$.
Now observe that since $\dim \num^{d-1}(X_{n+1}) - \dim \num^{d-1}(X_{n})$
is equal to the number of $\pi_n$-exceptional irreducible divisors, 
for any class $\gamma\in \num^{d-1}(X_{n+1})$
 we have $\gamma = \pi_n^*(\pi_n)_* \gamma$ iff $(\gamma \cdot E)=0$ for any $\pi_n$-exceptional divisor (see also~\cite[Theorem~1.2]{MR3390043}. 
 We may thus find a $\pi_n$-exceptional divisor $E_n$ such that $(\beta_{X_{n+1}}\cdot E_n) \neq 0$. 
  
We may thus construct by induction a sequence of classes $\alpha_n\in \num^1(X_n)$, by setting $\alpha_0 = 0$ and 
$\alpha_{n+1} := \pi_n^*\alpha_n+ \frac1{(\beta_{X_{n+1}}\cdot E_n)} E_n$ so that $(\beta_{X_{n+1}}\cdot\alpha_{n+1}) = (\beta_{X_{n}}\cdot\alpha_{n})+1$. 

We claim that the sequence of Cartier $b$-divisor classes $[\alpha_n]$ is weakly converging to a class in $\alpha \in \wnum^1(\X)$.
We briefly justify this fact. 

Pick any smooth model $\pi\colon X'\to X$. Let $\cV_{X'}$ be the set of $\pi$-exceptional divisors. 
This set can be canonically identified with the set of divisorial valuations on the fraction field 
$K(X)$ having a codimension $1$ center in $X'$, and a codimension at least $2$ center in $X$. 
Any class $\gamma \in \num^1(X')$ can be decomposed in a unique way as follows: 
$\gamma = \pi^* \gamma_X + \sum_{E\in\cV_{X'}} \ord_E(\gamma) [E]$
where $\gamma_X \in \num^1(X)$ and $a_E \in \R$. 

Since $(\alpha_n)_X=0$ for all $n$, we have the representation $(\alpha_n)_{X'} = \sum_{E\in\mathcal{V}_{X'}}  \ord_E(\alpha_n) [E]$.
The key observation is that $(\cup_n \cV_{X_n})\cap \cV_{X'}$ is finite. 
Since $\cV_{X_n}$ is increasing, for $n_0$ large enough,  $(\cup_n \cV_{X_n})\cap \cV_{X'}= \cV_{X_{n_0}}\cap \cV_{X'}$
therefore $(\alpha_n)_{X'} =(\alpha_{n_0})_{X'}$ for all $n\ge n_0$ which
proves the claim.

\smallskip

Finally by weak continuity of $\ell$,  we have 
\[\ell(\alpha) = \lim_n\ell([\alpha_n])= (\beta_{X_{n}}\cdot\alpha_{n}) \to \infty\] 
which is absurd.
\end{proof}


\subsection{Lehmann-Xiao's decompositions of curve classes}\label{sec:LX}

Our objective is to relate our Theorem~\ref{thm:MA} to the results of Lehmann and Xiao~\cite{MATH06656980}.
Let us review briefly their construction. 

Let $\alpha$ be any big class in $\num^{d-1}(X')$ for some model $X'$ so that $\alpha \ge c \om^{d-1}$ for some $c>0$.
Then \cite[Theorem 1.3]{MATH06656980} states the existence of a unique class $\LX(\alpha)\in\Nef(X')$ such that
\[
\alpha \ge \LX(\alpha)^{d-1} \text{ and } \LX(\alpha) \cdot \left(\alpha - \LX(\alpha)^{d-1}\right) =0~.
\]

Our next result relies on the  extension to $b$-classes of the functional $\hvol$ introduced in~\cite{zbMATH06798253}.
Recall that for any pseudo-effective class $\alpha\in\num^{d-1}(X')$, Xiao defined the following quantity:
\begin{equation}\label{def:hat-vol-model}
\hvol_{X'}(\alpha) :=\left(\inf_{\substack{\beta \in\Nef(X')\\ (\beta^d)>0}} \, \frac{(\alpha\cdot \beta)}{(\beta^d)^{1/d}}\right)^{d/d-1}\in \R_+~.
\end{equation}
We extend this functional to any pseudo-effective class $\alpha\in\Vect(\BPF^{d-1}(\X))$ by setting
\begin{equation}\label{def:hat-vol}
\hvol(\alpha) :=\left(\inf_{\substack{\beta \in\Nef(\X)\\ (\beta^d)>0}} \, \frac{(\alpha\cdot \beta)}{(\beta^d)^{1/d}}\right)^{d/d-1}\in \R_+~.
\end{equation}

\begin{thm}\label{thm:relate-to-LX}
For any big curve class $\alpha\in\BPF^{d-1}(\X)$, 
the family of Cartier $b$-divisor classes $\left[\LX(\alpha_{X'})\right]$
converges weakly to $\beta$ which is the unique nef $b$-divisor class satisfying
$\alpha = (\beta^{d-1})$.

 In other words, for any $\gamma \in\cnum^{d-1}(\X)$ and for any $\epsilon>0$, there exists a model $X'$ such that for any
model $X''\ge X'$ we have
\[
\left|
\left(\left[\LX(\alpha_{X''})\right] \cdot \gamma\right)- (\beta\cdot\gamma)
\right|
\le \epsilon
~.\]
\end{thm}

\begin{proof}

Suppose that $\alpha\in\BPF^{d-1}(\X)$ is big. By Theorem~\ref{thm:MA}, 
one can write $\alpha = \gamma^{d-1}$ for some big curve class $\gamma\in\Nef(\X)$. We claim that $\gamma$
is the unique minimizer computing $\hvol(\alpha)$.

Pick any class $\beta\in\Nef(\X)$ such that $(\beta^d)>0$, and normalize it by $(\beta\cdot\om^{d-1}) =+1$.  
By the proof of Theorem~\ref{thm:MA}, the class $\gamma/(\alpha\cdot \gamma)$ is the unique class in 
$\{\beta'\in\Nef(\X),\, (\beta'\cdot\om^{d-1}) =+1\}$ maximizing $(\beta^{'d})$.
It follows that 
\[
(\beta^d)\le \frac{(\gamma^d)}{(\alpha\cdot \gamma)^d}
=(\gamma^d)^{1-d}
\]
hence by Proposition~\ref{prop:KT}, we get 
\[
\frac{(\alpha\cdot \beta)}{(\beta^d)^{1/d}}
=
\frac{(\gamma^{d-1}\cdot \beta)}{(\beta^d)^{1/d}}
\ge
(\gamma^d)^{(d-1)/d}
=
\frac{(\alpha\cdot \gamma)}{(\gamma^d)^{1/d}}
\]
which proves the claim. 

\smallskip

For any model $X'$, we have
\begin{align*}
\hvol_{X'}(\alpha_{X'}) 
=
\left(\inf_{\substack{\beta \in \Nef(X')\\ (\beta^d)>0}} \, \frac{(\alpha_{X'}\cdot \beta)}{(\beta^d)^{1/d}}\right)^{d/d-1}
&=
\left(\inf_{\substack{\beta \in \Nef(X')\\ (\beta^d)>0}} \, \frac{(\alpha\cdot [\beta_{X'}])}{(\beta^d)^{1/d}}\right)^{d/d-1}
\\
&\ge
\left(\inf_{\substack{\beta \in \Nef(\X)\\ (\beta^d)>0}} \, \frac{(\alpha \cdot \beta)}{(\beta^d)^{1/d}}\right)^{d/d-1}
= \hvol(\alpha)
~.\end{align*}
\begin{lem}\label{lem:existence-model}
For any $\epsilon>0$, there exists a model $X'$ such that for any $X''\ge X'$ we have
$\hvol_{X''}(\alpha_{X''}) \le\hvol(\alpha) +\epsilon$.
 \end{lem}
Fix any model $X'$ as in the lemma. Since $\alpha_{X'}$ is big, it follows from~\cite[Theorem 1.3]{MATH06656980}
that $\LX(\alpha_{X'})$ is the unique minimizer for $\hvol_{X'}(\alpha_{X'})$.
From the previous lemma, and the relation 
$(\alpha_{X'}- \LX(\alpha_{X'})^{d-1})\cdot\LX(\alpha_{X'})=0$, we get
\begin{align}
\left|
(\gamma^d)- \left(\LX(\alpha_{X'})^d\right)
\right|
&=
\left|\hvol(\alpha)- \hvol_{X''}(\alpha_{X''})\right|  \le \epsilon
\\
|\left(\gamma^{d-1}\cdot \LX(\alpha_{X'})\right)
- 
\left(\LX(\alpha_{X'})^d\right)|
&=
|\left(\alpha\cdot \LX(\alpha_{X'})\right)
- 
\left(\LX(\alpha_{X'})^d\right)|
=0~. \label{eq;6665}
\end{align}
We conclude using~\eqref{eq:Diskant}.
\end{proof}

\begin{proof}[Proof of Lemma~\ref{lem:existence-model}]
Pick any sequence $\gamma_n\in\cNef(\X)$ decreasing to $\gamma$.
By Theorem~\ref{thm:intersection} (3), we have
\[
\lim_{n\to\infty} \left( \frac{(\alpha\cdot \gamma_n)}{(\gamma_n^d)^{1/d}}\right)^{d/d-1} = \hvol(\alpha)
~.\]
Choose $n$ large enough such that 
\[
\left( \frac{(\alpha\cdot \gamma_n)}{(\gamma_n^d)^{1/d}}\right)^{d/d-1} \le  \hvol(\alpha) + \epsilon
~.\]
Then in any model $X''$ dominating a model in which $\gamma_n$ is determined, we get 
$\hvol_{X''}(\alpha_{X''})\le  \hvol(\alpha) + \epsilon$.
\end{proof}

\begin{rem}
Suppose that $\alpha_n \to \alpha \ge \om^{d-1}$. Then we can write $\alpha_n =\gamma_n^{d-1}$ and $\alpha=\gamma^{d-1}$ for some $\gamma_n , \gamma\in\Nef(\X)$. If 
$\hvol(\gamma_n)\to\hvol(\gamma)$, then it is possible to show that $\gamma_n \to \gamma$.
It would be interesting to find more general criteria ensuring the convergence  $\gamma_n \to \gamma$.
\end{rem}

We conclude this section by relating the big curve b-classes with the positivity of the functional $\hvol$. The following result is a natural generalization of \cite[Theorem 5.2]{MATH06656980} to b-classes.
A similar statement holds in codimension $1$, which is a direct application of Siu's inequalities.

\begin{thm}\label{thm:char-of-bigness}
Pick any curve $b$-numerical class $\alpha\in\BPF^{d-1}(\X)$, and let $\om$ be any big nef Cartier $b$-divisor class.
Then the following are equivalent:
\begin{enumerate}
\item
$\hvol(\alpha) >0$; 
\item
$\alpha$ is big.
\end{enumerate}
\end{thm}

\begin{proof}
The implication $(2) \Rightarrow (1)$ is easy since $\hvol$ is increasing. 
Suppose conversely that $\hvol(\alpha)>0$. Pick $\epsilon>0$ and write $\alpha_\epsilon = \alpha+ \epsilon \om^{d-1}$.
We apply~\cite[Theorem~5.19]{MATH06656980} to the big class $(\alpha_\epsilon)_{X'}$ and the movable class $\om^{d-1}$. 
Write $B := \LX([(\alpha_\epsilon)_{X'}])\in\Nef(X')$ so that $(B^d)= \hvol_{X'}((\alpha_\epsilon)_{X'})$.
We obtain
\[
[(\alpha_\epsilon)_{X'}]
\ge
 \frac{\hvol([(\alpha_\epsilon)_{X'}])}{d (B\cdot \om^{d-1})}
\, \om^{d-1}\]
By Khovanskii-Teissier's inequalities, we get $(B\cdot \om^{d-1}) \le \frac{(B^{d-1}\cdot \om)^{d-1}}{(B^d)^{d-2}}$
hence
\[
[(\alpha_\epsilon)_{X'}]
\ge
\om^{d-1} \, \frac{(B^d)^{d-1}}{d (B^{d-1}\cdot \om)^{d-1}}
\mathop{\ge}^{\eqref{eq;6665}}
\om^{d-1} \, \frac{\hvol(\alpha_\epsilon)^{d-1}}{d (\alpha_\epsilon\cdot \om)^{d-1}}
\]
since $B^{d-1}= \LX([(\alpha_\epsilon)_{X'}])^{d-1} \le (\alpha_\epsilon)_{X'}$.
Now  $[(\alpha_\epsilon)_{X'}]$ converges weakly to $\alpha_\epsilon$, 
so that
\[
\alpha_\epsilon
\ge
\om^{d-1} \, \frac{\hvol(\alpha_\epsilon)^{d-1}}{d (\alpha_\epsilon\cdot \om)^{d-1}}
\ge
\om^{d-1} \, \frac{\hvol(\alpha)^{d-1}}{d (\alpha_\epsilon\cdot \om)^{d-1}}~.
\]
Letting $\epsilon\to0$, we get 
$\alpha\ge \om^{d-1} \, \frac{\hvol(\alpha)^{d-1}}{d (\alpha\cdot \om)^{d-1}}$ which concludes the proof.
\end{proof}

\subsection{Positive intersections and $(d-1)$-th root}\label{sec:d-root}
Xiao associated to any movable class $\alpha\in\num^{d-1}(X)$ the following quantity:
\begin{equation}\label{eq:def-mov}
\fM_{X}(\alpha) = \inf_{\substack{\gamma\in \num^1(X)\\ \text{big and movable}}} \left(
\frac{(\gamma\cdot\alpha)}{\vol(\gamma)^{1/d}}
\right)^{d/(d-1)}
~,\end{equation}
where $\vol(\gamma)$ is the standard volume function on the cone of big divisor classes 
(which is defined by $\vol(\gamma):=\langle P(\gamma)^d\rangle$).

Lehmann and Xiao then proved that if $\fM_{X}(\alpha)>0$, then there exists a unique class $\Mv_{X}(\alpha)\in\num^1(X')$ minimizing the right hand side of~\eqref{eq:def-mov}, and that this class 
satisfies $\alpha_X = \langle P([\Mv_{X}(\alpha)])^{d-1}\rangle_{X}$, see~\cite[Theorem~3.14]{lehmann_xiao_positivity}.

\smallskip

If  $\alpha\in\BPF^{d-1}(X)$ is a $b$-numerical curve class, then $\alpha_{X'}$ is movable for any model $X'$ and we may set
$\fM_{X'}(\alpha):=\fM_{X'}(\alpha_{X'})$. 

\begin{thm}\label{thm:LX-Mfnt}
Pick any big curve class $\alpha\in\BPF^{d-1}(\X)$.
Then for any model $X'$, the class $\alpha_{X'}$ is movable and $\fM_{X'}(\alpha_{X'})>0$.
The family of Cartier $b$-divisor classes $\left[\Mv_{X'}(\alpha)\right]$
converges weakly to $\beta$ which is the unique nef $b$-divisor class satisfying
$\alpha = (\beta^{d-1})$.
\end{thm}

\begin{proof}
Choose any big curve class $\alpha\in\BPF^{d-1}(\X)$, and write $\alpha = (\beta^{d-1})$  with $\beta\in\Nef(\X)$.
Pick any model $X'$ and any big and movable class $\gamma\in\num^1(X')$.
Since $\gamma$ is movable, we have $P([\gamma])_{X'} =\gamma$, and $\vol(P([\gamma]))=\vol(\gamma)$.
We thus obtain
\begin{align*}
\left(
\frac{(\gamma\cdot\alpha_{X'})}{\vol(\gamma)^{1/d}}
\right)^{d/(d-1)}
=
 \left(
\frac{(P([\gamma])\cdot[\alpha_{X'}])}{\vol(P([\gamma]))^{1/d}}
\right)^{d/(d-1)}
\mathop{\ge}\limits^{\text{Lemma~\ref{lem:negativity}}} 
 \left(
\frac{(P([\gamma])\cdot\alpha)}{\vol(P([\gamma]))^{1/d}}
\right)^{d/(d-1)}
\ge
\hvol(\alpha)~,
\end{align*}
which proves $\fM_{X'}(\alpha) \ge \hvol(\alpha) >0$.

On the other hand, we claim that: 
\begin{equation}\label{eq:1234}
\lim_{X'} \left(
\frac{([\beta_{X'}]\cdot\alpha)}{\vol(\beta_{X'})^{1/d}}
\right)^{d/(d-1)}
=\hvol(\alpha)
~.\end{equation}
Indeed, since $\vol(\beta_{X'})=\langle P([\beta_{X'}])^d\rangle$ and $ P([\beta_{X'}])$ is a decreasing net of nef divisor classes converging to $\beta$, 
we have $\vol(\beta_{X'})\to \vol(\beta)$ by Theorem~\ref{thm:intersection}.
We also have $([\beta_{X'}]\cdot\alpha)\to (\beta\cdot\alpha)$ as $X'$ tends to infinity in the net of all models. 
Indeed this fact is obvious when $\alpha$ is Cartier, and the general case follows from 
Corollary~\ref{cor:approx-curve}. This proves~\eqref{eq:1234}.
Since 
\[
 \fM_{X'}(\alpha)\le
 \left(
\frac{([\beta_{X'}]\cdot\alpha)}{\vol(\beta_{X'})^{1/d}}
\right)^{d/(d-1)}
\]
we conclude that $\lim_{X'} \fM_{X'}(\alpha) = \hvol(\alpha)$.

We get
\[
|(\beta^d)- (P([\Mv_{X'}(\alpha)])^d)| = |\hvol(\alpha) - \fM_{X'}(\alpha)| \to 0
\]
and
\begin{align*}
(\beta^{d-1}\cdot P([\Mv_{X'}(\alpha)]))
&= 
(\alpha\cdot P([\Mv_{X'}(\alpha)]))
\le([\alpha_{X'}]\cdot P([\Mv_{X'}(\alpha)]))
\\
\le 
(P([\Mv_{X'}(\alpha)])^d),
\end{align*}
and Diskant's estimates~\eqref{eq:Diskant} imply $P([\Mv_{X'}(\alpha)])\to \beta$.
Since $P([\Mv_{X'}(\alpha)])_{X'} = \Mv_{X'}(\alpha)$, it follows that
 $[\Mv_{X'}(\alpha)]\to \beta$. 
\end{proof}

\begin{rem}
Fulger and Lehmann~\cite{fulger_lehmann_zariski} have also introduced a notion of Zariski decomposition of a big class $\alpha \in\num^k(X)$ for any $k$.
This decomposition is based on the mobility functional which is not invariant under birational pull-backs. We dont know how to extend this functional to the space of Cartier $b$-numerical cycles, 
not to mention to the BPF cone. 
\end{rem}


\bibliographystyle{alpha}
\bibliography{biblio}
\addcontentsline{toc}{section}{References}

\end{document}